\documentclass[reqno,oneside,11pt,letterpaper]{amsart}

\usepackage[english]{babel}
\usepackage[utf8]{inputenc}
\usepackage{amsfonts,amsmath,amsthm}
\usepackage{stmaryrd}
\usepackage{thmtools,thm-restate}
\usepackage[T1]{fontenc}
\usepackage{enumitem}
\usepackage{varioref}
\usepackage[all]{xy}
\usepackage{units}
\usepackage{hyperref}
\usepackage{graphicx}
\usepackage{amssymb}
\usepackage{mathabx}
\usepackage{times,mathptm, mathtools}
\usepackage{etoolbox}
\usepackage{tikz-cd}
\usepackage{mathrsfs}
\usepackage{mathdots}
\usepackage{fullpage}

\docsvlist{A,B,C,D,E,F,G,H,I,J,K,L,M,N,O,P,Q,R,S,T,U,V,W,X,Y,Z}

\docsvlist{A,B,C,D,E,F,G,H,I,J,K,L,M,N,O,P,Q,R,S,T,U,V,W,X,Y,Z}

\docsvlist{A,B,C,D,E,F,G,H,I,J,K,L,M,N,O,P,Q,R,S,T,U,V,W,X,Y,Z}

\docsvlist{A,B,C,D,E,F,G,H,I,J,K,L,M,N,O,P,Q,R,S,T,U,V,W,X,Y,Z}

\docsvlist{A,B,C,D,E,F,G,H,I,J,K,L,M,N,O,P,Q,R,S,T,U,V,W,X,Y,Z}

\renewcommand{\hat}{\widehat}
\newcommand{\R}{\mathbf{R}}
\newcommand{\C}{\mathbf{C}}
\newcommand{\Q}{\mathbf{Q}}
\newcommand{\Z}{\mathbf{Z}}
\newcommand{\N}{\mathbf{N}}
\newcommand{\A}{\mathbf{A}}

\newcommand{\K}{\mathbf{K}}
\renewcommand{\P}{\mathbf{P}}

\newcommand{\m}{\mathfrak{m}}

\newcommand{\G}{\mathbb G}

\renewcommand{\div}{\operatorname{div}}
\DeclareMathOperator{\QAlb}{QAlb}
\DeclareMathOperator{\Pic}{Pic}
\DeclareMathOperator{\Div}{Div}
\DeclareMathOperator{\Gal}{Gal}
\DeclareMathOperator{\ord}{ord}
\newcommand{\OO}{\mathcal O}
\DeclareMathOperator{\id}{id}

\DeclareMathOperator{\fin}{f}

\DeclareMathOperator{\Supp}{Supp}
\DeclareMathOperator{\supp}{Supp}

\renewcommand{\mod}{\text{mod}}
\DeclareMathOperator{\NS}{NS}

\DeclareMathOperator{\Ind}{Ind}

\DeclareMathOperator{\Per}{Per}

\DeclareMathOperator{\spec}{Spec}

\DeclareMathOperator{\Aut}{Aut}
\DeclareMathOperator{\SL}{SL}
\DeclareMathOperator{\GL}{GL}

\newcommand{\DivInf}{\Div_\infty}
\DeclareMathOperator{\car}{char}
\DeclareMathOperator{\an}{an}

\DeclareMathOperator{\BD}{\partial_X S_F}
\DeclareMathOperator{\tr.deg}{tr.deg}
\DeclareMathOperator{\Cinf}{Cartier_\infty}
\DeclareMathOperator{\Winf}{Weil_\infty}
\DeclareMathOperator{\cNS}{c-NS}
\DeclareMathOperator{\wNS}{w-NS}

\newtheorem{thm}{Theorem}[section]
\newtheorem{thm*}{Theorem}
\newtheorem{bigthm}{Theorem}
\newtheorem{prop}[thm]{Proposition}
\newtheorem{cor}[thm]{Corollary}
\newtheorem{lemme}[thm]{Lemma}
\newtheorem{conj}[thm]{Conjecture}
\newtheorem{claim}[thm]{Claim}
\theoremstyle{definition}
\newtheorem{dfn}[thm]{Definition}
\newtheorem{ex}[thm]{Example}
\newtheorem{rmq}[thm]{Remark}

\AtBeginEnvironment{dfn}{%
  \setlist[enumerate]{label={(\roman*)}}
}

\AtBeginEnvironment{thm}{%
  \setlist[enumerate,1]{label={(\arabic*)}}
}

\AtBeginEnvironment{prop}{%
  \setlist[enumerate,1]{label={(\arabic*)}}
}

\AtBeginEnvironment{lemme}{%
  \setlist[enumerate,1]{label={(\arabic*)}}
}

\begin{document}
\author{Marc Abboud \\ Université de Neuchâtel}
\title{Rigidity of periodic points for loxodromic automorphisms of affine surfaces}
\address{Marc Abboud, Institut de mathématiques, Université de Neuchâtel
\\ Rue Emile-Argand 11 CH-2000 Neuchâtel}
\email{marc.abboud@normalesup.org}
\thanks{The author acknowledge support by the Swiss National Science Foundation Grant “Birational transformations of higher dimensional varieties” 200020-214999.}
\maketitle

\begin{abstract}
   We show that two automorphisms of an affine surface with dynamical degree $>1$ share a Zariski dense set of periodic
   points if and only if they have the same periodic points. We construct canonical heights for these automorphisms and
   use arithmetic equidistribution for adelic line bundles over quasiprojective varieties following the work of Yuan and
   Zhang. When the base field is not a number field or the function field of a curve we use the theory of Moriwaki
   heights to prove the result.
\end{abstract}
\tableofcontents
\section{Introduction}
\subsection{Loxodromic automorphisms}\label{subsec:loxodromic-automorphisms}
Let $F$ be a field, a \emph{variety} over $F$ is a normal separated integral scheme over $F$.
An affine surface is an affine variety of dimension 2. Let $S_F$ be an affine surface over $F$. A \emph{completion} of
$S_F$ is a projective surface $X$ over $F$
together with an open embedding $S_F \hookrightarrow X$. If $f$ is an automorphism of $S_F$, then $f$ defines a
birational selfmap on every completion of $S_F$. The first dynamical degree $\lambda_1 (f)$ is defined as the following limit
\begin{equation}
  \lambda_1 (f) := \lim_N \left( \left(f^N\right)^* H \cdot H \right)^{1/N}
  \label{eq:}
\end{equation}
where $H$ is an ample divisor over a completion $X$ of $S_F$. The limit is well defined and does not depend on the
choice of $X$ nor $H$. We have that $\lambda_1 (f) \geq 1$ and an automorphism with $\lambda_1(f) > 1$ is called
\emph{loxodromic}.

\subsection{Rigdity of periodic points}\label{subsec:rigidity-periodic-points}
If $f \in \Aut (S_F)$, we write $\Per(f)$ for the set of $\overline F$-periodic points of $f$ in
$S_F$ (we will show in this paper that any periodic point of $f$ must be defined over $\overline F$ if $f$ is
loxodromic). The main result of
this paper is the following
\begin{bigthm}\label{bigthm:rigidity-periodic-points}
  Let $S_F$ be a normal affine surface over a field $F$ and let $f,g \in \Aut(S_F)$ be two loxodromic automorphisms,
  then we have the following equivalence
  \begin{enumerate}
    \item $\Per(f) \cap \Per(g)$ is Zariski dense,
    \item $\Per (f) = \Per(g)$.
  \end{enumerate}
\end{bigthm}
The proof is based on arithmetic dynamics techniques. We construct Green functions and canonical heights for loxodromic
automorphisms of affine surfaces. This work is based on a previous work of the author
\cite{abboudDynamicsEndomorphismsAffine2023} where the dynamics of loxodromic automorphisms of affine surfaces is
classified.

This kind of results are called \emph{unlikely intersections} result in the litterature. Baker and DeMarco in
\cite{bakerPreperiodicPointsUnlikely2011} showed this theorem for
rational transformations of $\P^1$ over $\C$ and Yuan and Zhang showed it in \cite{yuanArithmeticHodgeIndex2017,
yuanArithmeticHodgeIndex2021} for polarised endomorphisms of projective varieties over any field. Carney then
generalised the result of Yuan and Zhang to positive characteristic in \cite{carneyArithmeticHodgeIndex2020}. In
\cite{cantatInvariantMeasuresLarge2023} Cantat and Dujardin showed the following related results: If $X$ is a complex projective
surface and $\Gamma \subset \Aut (X)$ is a large subgroup of automorphism then the set of finite orbits of $\Gamma$ is
not Zariski dense unless $X$ is a Kummer surface, i.e a finite equivariant quotient of an abelian surface.

The first instance of such a result for non projective varieties is the result of Dujardin and Favre from
\cite{dujardinDynamicalManinMumford2017} for Hénon maps of
the affine plane. They showed that if $F$ is a number field, then $\Per(f) \cap \Per(g)$ is Zariski dense if and only if
$f,g$ share a common iterate. We will call this equivalence \emph{strong rigidity of periodic points}. The authors show
in loc.cit that strong rigidity
of periodic points also holds for Hénon maps over $\C$ if the Jacobian of $f$ is not a root of unity.
\subsection{The counterexample of the algebraic torus}\label{subsec:counterexample-torus}
We cannot expect that strong rigidity of periodic points holds for any affine surface. Indeed suppose $S_F = \G_m^2$ is the
algebraic torus. Every matrix $A = \begin{pmatrix}a & b \\c & d \end{pmatrix}$ in $\SL_2(\Z)$ defines a monomial automorphism
\begin{equation}
  f_A (x,y) = \left( x^a y^b, x^c y^b \right)
  \label{eq:}
\end{equation}
and $\lambda_1(f_A) = \rho (A)$ the spectral radius of $A$. For any loxodromic monomial automorphism $f$ of $\G_m^2$,
$\Per(f) = \mathbb U \times \mathbb U$ where $\mathbb U$ is the set of roots of unity in $\overline F$. Thus, any two
loxodromic monomial automorphisms of $\G_m^2$ have the same periodic points but do not necessarily share a common iterate.

Consider the involution $\sigma (x,y) = (x^{-1}, y^{-1})$, then $\sigma$ commutes with any monomial automorphims. The
quotient $\G_m^2 / \langle \sigma \rangle$ is a normal affine surface that we denote by $M_4$ (this notation will be
justified in the following paragraph). Every monomial loxodromic automorphism of $\G_m^2$ descends to a loxodromic
automorphism over $M_4$. Thus, the normal affine surface $M_4$ yields also a counterexample to the strong rigidity of
periodic points. we conjecture the following

\begin{conj}\label{conj:strong-rigidity}
  Let $F$ be a field of characteristic zero and $S_F$ a normal affine surface over $F$. If $f,g \in \Aut(S_F)$ are
  loxodromic automorphisms,
  then the following are equivalent
  \begin{enumerate}
    \item $\Per(f) \cap \Per(g)$ is Zariski dense.
    \item There exists $N,M \in \Z \setminus \left\{ 0 \right\}$ such that $f^N = g^M$.
  \end{enumerate}
  Unless, $S_F$ is an equivariant quotient of $\G_m^2$.
\end{conj}
This is the affine counterpart of the result of Cantat and Dujardin for projective surfaces.

\subsection{Strong rigidity of periodic points, results towards Conjecture
\ref{conj:strong-rigidity}}\label{subsec:strong-rigidity-intro}
We manage to show strong rigidity of periodic points for certain normal affine surfaces. First for Hénon maps, we
strengthen the result of Dujardin and Favre by removing the condition on the Jacobian
\begin{bigthm}\label{bigthm:strong-rigidity-henon}
  Let $f,g$ be two Hénon maps over $\C$, then $\Per(f) \cap \Per(g)$ is Zariski dense in $\C^2$ if and only if there
  exists $N,M \neq 0$ such that $f^N = g^M$.
\end{bigthm}

Let us introduce a family of affine surfaces: the family of Markov surfaces $\cM_D$. Let $D \in \C$.
The Markov surface $\cM_D$ is defined as the affine surface in $\A^3_\C$ given by the following equation
\begin{equation}
  x^2 + y^2 + z^2 = xyz +D.
  \label{eq:<+label+>}
\end{equation}
This family of surfaces has been well studied (see for example the introduction of \cite{cantatBersHenonPainleve2009}
and \cite{abboudUnlikelyIntersectionsProblem2024}). The automorphism group of $\cM_D$ does not depend on the parameter $D$
and is equal to $\SL_2(\Z)$ up to finite index. The parameter $D = 4$ is very peculiar as it admits a 2:1 cover
\begin{equation}
  (u,v) \in \G_m^2 \mapsto \left( u+ \frac{1}{u}, v + \frac{1}{v}, uv + \frac{1}{uv} \right) \in M_4.
  \label{eq:<+label+>}
\end{equation}
This map is exactly the quotient map $\G_m^2 \rightarrow \G_m^2 / \langle \sigma \rangle$. The parameter $D=4$ is the
only one where $\cM_D$ is a finite quotient of $\G_m^2$. In \cite{abboudUnlikelyIntersectionsProblem2024}, the author
shows the property of strong rigidity of periodic points for the surface $\cM_D$ defined over $\Q(D)$ for certain
algebraic parameters $D \in \overline \Q$. Theorem B of \cite{abboudUnlikelyIntersectionsProblem2024} along with Theorem
\ref{bigthm:rigidity-periodic-points} of this paper imply the strong rigidity of periodic points for transcendental
parameters.

\begin{bigthm}\label{bigthm:strong-rigidity-markov}
  Let $D\in \C$ be transcendental. If $f,g \in \Aut(\cM_D)$ are loxodromic automorphisms, then $\Per(f) \cap
  \Per(g)$ is Zariski dense if and only if there exists $N,M \neq 0$ such that $f^N = g^M$.
\end{bigthm}

\subsection{Plan of the paper}\label{subsec:plan-paper}
This paper contains two parts. First, we recall the notion of adelic line bundles and adelic divisors over a
quasiprojective variety. When $F$ is not algebraic over $\Q$ or is of positive characteristic we will use the notion of
Moriwaki heights. The main idea is that Moriwaki heights can be interpreted as an integral of local intersection
numbers. This corresponds to Sections \ref{sec:terminology} to \S\ref{sec:geometric-setting}.

In the second part, we prove the main theorems stated in this introduction. To do so, we will prove that a loxodromic
automorphisms admits two invariants adelic divisors
and use the arithmetic equidistribution theorem from \cite{yuanAdelicLineBundles2023}. The main difference with former
proofs of such result is that if the base field is not a number field or a function field with transcendence degree one
we will use Moriwaki heights instead of a specialisation argument. The reason is that a set of points which is
Zariski dense in the generic fiber might never be Zariski dense after specialisation (in dimension $\geq 2$).
\subsection*{Acknowledgments}
Part of this work was done during my PhD thesis. I would like to thank my PhD advisors Serge Cantat and Junyi Xie for their
guidance. I also thank Xinyi Yuan for
our discussions on adelic divisors. Part of this paper was
written during my visit at Beijing International Center for Mathematical Research which I thank for
its welcome. Finally, I thank the France 2030 framework programme Centre Henri
Lebesgue ANR-11-LABX-0020-01 and
European Research Council (ERCGOAT101053021) for creating an attractive mathematical environment.

\section{Terminology}\label{sec:terminology}
Let $R$ be an integral domain. A variety over $R$ is a flat, normal, integral, separated scheme over $\spec R$ locally of
finite type. If $R = \K$ is a field we require in addition that a variety over $\K$ is geometrically irreducible.

\subsection{Divisors and line bundles}\label{subsec:divisors}
Let $X$ be a normal $R$-variety, a \emph{Weil divisor} over $X$ is a formal sum of irreducible codimension 1 closed
subvarieties of $X$ with integer coefficients. A \emph{Cartier divisor} over $X$ is a global section of the sheaf
$\mathcal K_X^\times / \OO_X^\times$ where $\mathcal K_X$ is the sheaf of rational functions over $X$ and $\OO_X$ the
sheaf of regular functions over $X$. If $\A = \Q$ or $\R$, an $\A$-Cartier divisor is a formal sum $D = \sum_i a_i D_i$
where $a_i \in \A$ and $D_i$ is a Cartier divisor. An $\A$-Weil divisor is a formal sum of irreducible codimension 1
closed subvarieties with coefficients in $\A$. 

If $R$ is Noetherian (which will always be the case in this paper), the local ring at the generic point of an
irreducible codimension 1 closed subvariety is a Noetherian regular local ring of dimension 1 and thus a discrete
valuation ring. This implies that every $\R$-Cartier divisor induces a unique $\R$-Weil divisor. If $X$ is a projective
variety over a field $K$ and $D$ a Weil divisor over $X$, we write $\Gamma (X,D)$ for the set
\begin{equation}
  \Gamma (X,D) = \left\{ f \in K (X)^\times : \div (f) + D \geq 0 \right\}.
  \label{eq:<+label+>}
\end{equation}
If $X$ is a projective variety over $R$ and $D$ is a Cartier divisor, we write $\OO_X(D)$ for the line bundle associated to $D$.
If $R = \K$ is a field and $L$ is a line bundle over $X$. We write $H^0 (X,L)$ for the space of global sections of $L$. 

\subsection{Horizontal and vertical subvarieties}\label{subsec:horizontal-vertical-subvarieties}

Let $q : X \rightarrow Y$ be a
proper (e.g projective) morphism of varieties with generic fiber $q_\eta : X_\eta
\rightarrow \spec \eta$ with $\eta$ the generic point of $Y$. There are two types of irreducible closed subvarieties $Z
\subset X$ in $X$:
\begin{itemize}
  \item Horizontal ones which are the closure of an irreducible closed subvariety of the generic fiber
    $X_\eta$. They are characterized by $q(Z) = Y$.
  \item Vertical ones which are such that $Z \cap X_\eta = \emptyset$, they are characterized by the fact that
    $q(Z)$ is a strict closed subvariety of $Y$.
\end{itemize}
A closed subvariety is horizontal (resp. vertical) if all of its irreducible components are. An $\R$-Weil divisor is
horizontal (resp. vertical) if its support is so. Every $\R$-Weil divisor $D$ over $X$ can be uniquely split as a sum
$D = D_{hor} + D_{vert}$ where $D_{hor}$ is a horizontal $\R$-Weil divisor and $D_{vert}$ is a horizontal one. 
Let $Z \rightarrow X$ be a blowup of $X$. We say that the blow-up is \emph{horizontal} (resp. vertical) if its center is
a horizontal (resp. vertical) subvariety.

Finally, if $U \subset Y$ is an open subset and $Z \subset X$ is a closed subvariety. We say that $Z$ \emph{lies} above
$U$, if $q (Z) \cap U \neq \emptyset$.

\subsection{Models}\label{subsec:models}
We follow the definitions from \S 2.3 of \cite{yuanAdelicLineBundles2023}.
A morphism of schemes $i : X \rightarrow Y$ is a \emph{pro-open} immersion if
\begin{enumerate}
  \item $i$ is injective on the topological spaces.
  \item for any $x \in X$, the induced map $i^* : \OO_{Y, i(x)} \rightarrow \OO_{X,x}$ is an isomorphism.
\end{enumerate}

Let $F$ be a field and let $R$ be an integral domain with an injective ring homomorphism $R \hookrightarrow F$. Let $X_F$
be a quasiprojective variety over $F$. A \emph{quasiprojective model} (resp. \emph{projective model}) of $X_F$ over
$R$ is a quasiprojective (resp. projective) variety $Y_R$ over $R$ with a pro-open immersion $X_F \rightarrow Y_R$. We
are mainly interested in the case where $F$ is a finitely generated field over the fraction field of $R$ (We say that
$F$ is finitely generated over $R$). In particular, in that case a quasiprojective model of $\spec F$ over $R$ is a
quasiprojective variety over $R$ with function field isomorphic to $F$.

If $F$ is a field and $U_F$ a quasiprojective variety over $F$, we will call a projective model of $U_F$ over $F$ a
\emph{completion} of $U_F$.

Suppose $F$ is a finitely generated field over a field $k$ with $\tr.deg F/k =d \geq 1$. Let $B_k$ be a projective model
of $F$ over $k$. If $U_F$ is a quasiprojective variety over $F$ a \emph{quasiprojective} (resp. \emph{projective}) model
of $U_F$ over $B$ is a quasiprojective (resp. projective) variety $X_B$ over $k$ with a morphism of varieties $X_B
\rightarrow B$.
\begin{lemme}[Lemma 2.3.3 of \cite{yuanAdelicLineBundles2023}]\label{lemme:structure-quasiprojective-model}
  Let $F$ be a finitely generated field over an integral domain $R$. Let $U_F$ be a quasiprojective variety over $F$ and
  $X_R$ a quasiprojective model of $U_F$ over $R$. Then, there exists an open subscheme $X_R '$ of $X_R$ and a flat
  morphism $q : X_R ' \rightarrow V_R$ of quasiprojective varieties over $R$ such that the generic fiber is
  isomorphic to $U_F \rightarrow \spec F$.

  Furthermore, if $U_F$ is projective, then we can assume that $q$ is projective.
\end{lemme}

\begin{lemme}[Lemma 2.3.4 of \cite{yuanAdelicLineBundles2023}]
  \label{lemme:system-quasiprojective-models}
  Let $F$ be a finitely generated field over an integral domain $R$. If $U_F$ is a quasiprojective variety over $F$ and
  $X_R$ a quasiprojective model of $U_F$ over $R$, then
  the inverse system of open neighbourhoods of $U_F$ in $X_R$ is cofinal in the system of quasiprojective
  models of $U_F$ over $R$.
\end{lemme}
\begin{proof}
  Let $X_R'$ be another quasiprojective model of $U_F$ over $R$. Then, the birational map
  $X_R \dashrightarrow X_K'$ induces an isomorphism between an open neighbourhood of $U_F$ in $X_R$ and an
  open subset of $X_R'$.
\end{proof}

As we will very often use models of varieties over different rings and fields, we will write varieties
with a subscript indicating over which ring or field they are defined unless it burdens the notations too much.

\subsection{Measure theory and topology}\label{subsec:measures}
Let $\Omega$ be a measurable space with a positive measure $\mu$. We say that a measurable subset $A \subset \Omega$ is
of full measure if $\mu (\Omega \setminus A) = 0$. If $(A_n)_{n \geq 0}$ is a sequence of measurable sets of full
measure, then $\bigcap_n A_n$ is also a subset of full measure.
We say that a property is true for $\mu$-almost every $w \in \Omega$ if it holds for any $w \in A$ where $A \subset
\Omega$ is of full measure.

\begin{dfn}
Let $T$ be a locally compact Hausdorff space with its Borel $\sigma$-algebra. A \emph{Radon measure} over $T$ is a
positive measure $\mu$ on the Borel $\sigma$-algebra such that
\begin{enumerate}
  \item For every Borel set $V \subset T, \mu (V) = \sup_{K \subset V} \mu (K)$ where $K$ runs through compact subsets.
  \item For every Borel set $V \subset T, \mu(V) = \inf_{U \supset V} \mu (U)$ where $U$ runs through open subsets.
  \item For every $t \in T$, there exists an open neighbourhood $U$ of $t$ such that $\mu (U) < + \infty$.
\end{enumerate}
\end{dfn}

Finally, if $X$ is a topological space, then we write $\mathcal C^0 (X, \R)$ for the set of continuous function $X \rightarrow
\R$ and $\mathcal C^0_c (X, \R)$ for the set of continuous function $X \rightarrow \R$ with compact support.

\section{Analytification of algebraic varieties}
\label{sec:analyt-algebr-vari}
\subsection{Berkovich spaces}\label{subsec:berkovich-spaces} For a general reference on Berkovich spaces, we refer to
\cite{berkovichSpectralTheoryAnalytic2012}.
Let $\K_v$ be a complete field with respect to an absolute value $| \cdot |_v$. If $X_{K_v}$ is a quasiprojective
variety over $K_v$, we write $X_{\K_v}^{\an}$ for the Berkovich analytification of $X_{K_v}$ with respect to $K_v$. It is a locally ringed space with a contraction map
\begin{equation}
  c : X_{\K_v}^{\an} \rightarrow X_{\K_v}.
  \label{<+label+>}
\end{equation}
 It is a
Hausdorff space. In
particular, if $X_{\K_v}$ is proper (e.g projective), then $X_{\K_v}^{\an}$ is compact.

Let $\overline \K_v$ be an algebraic closure of $\K_v$. The absolute value $\left| \cdot \right|_v$ extends naturally to
$\overline \K_v$. If $p \in X_{\K_v}(\overline \K_v)$ is a rational point, then it defines a point in
$X_{\K_v}^{\an}$. We thus have a map
\begin{equation}
  \iota_0: X_{\K_v} (\overline \K_v) \rightarrow X_{\K_v}^{\an}
  \label{<+label+>}
\end{equation}
and we write $X_{\K_v} (\overline \K_v)$ for its image. It is a dense subset of $X_{\K_v}^{\an}$. This map is generally
not injective as two points $p,q \in X_{\K_v}(\overline \K_v)$ define the same seminorm if and only if they are in the
same orbit for the action of the Galois group $\Gal (\overline \K_v / \K_v)$.

If $\phi : X_{\K_v} \rightarrow Y_{\K_v}$ is a morphism of varieties, then there exists a unique morphism
\begin{equation}
  \phi^{\an} :
  X_{\K_v}^{\an} \rightarrow Y_{\K_v}^{\an}
\end{equation}
such that the diagram
\begin{equation}
\begin{tikzcd}
  X_{\K_v}^{\an} \ar[r, "\phi^{\an}"] \ar[d] & Y_{\K_v}^{\an} \ar[d] \\
  X_{\K_v} \ar[r, "\phi"] & Y_{\K_v}
  \label{<+label+>}
\end{tikzcd}
\end{equation}
commutes. In particular, if $X_{\K_v} \subset Y_{\K_v}$, then $X_{\K_v}^{\an}$ is isomorphic to $c_Y^{-1}(X_{\K_v})
\subset Y_{\K_v}^{\an}$.

If $\K_v$ is not algebraically closed, let $\C_v$ be the completion of the algebraic closure of $\K_v$ with respect to
$v$. If $X_{\K_v}$ is a variety over $\K_v$ and $X_{\C_v}$ is the base change to $\C_v$, then we have the following
relation for the Berkovich space
\begin{equation}
  X_{\K_v}^{\an} = X_{\C_v}^{\an} / \Gal(\overline \K_v / \K_v)
  \label{eq:<+label+>}
\end{equation}
and the continuous map $X_{\C_v}^{\an} \rightarrow X_{\K_v}^{\an}$ is proper (the preimage of a compact subset is a
compact subset) if $X_{\K_v}$ is quasiprojective.
In particular, if $\K_v = \R$ and $X_\R$ is a variety over $\R$, then
\begin{equation}
  X_\R^{\an} = X_\R (\C) / (z \mapsto \overline z).
  \label{eq:<+label+>}
\end{equation}

\subsection{Places and restricted analytic spaces}\label{subsec:places}
Let $\K$ be a number field. A \emph{place} $v$ of $\K$ is an equivalence class of absolute values over $\K$. If $v$ is
archimedean then there is an embedding $\sigma : \K \hookrightarrow \C$ such that any absolute value representing $v$ is
of the form $|x| = \left|\sigma(x) \right|^t_\C$ with $0 < t \leq 1$. In that case we will write $|\cdot|_v$ for the
absolute value with $t=1$. If $v$ is non-archimedean (we also say that $v$ is \emph{finite}) it
lies over a prime $p$ then we write $|\cdot|_v$ for the
absolute value of $\K$ representing $v$ such that $|p|_v = \frac{1}{p}$. Every finite place $v$ is of the form
\begin{equation}
  v(P) = \# \left( \OO_\K / \m \right)^{-\ord_\m (P)}
  \label{<+label+>}
\end{equation}
for $P \in \K$ where $\m$ is a maximal ideal of $\OO_\K$.

We write $\cM(\K)$ for the set of places of $\K$ and for every $v \in \cM (\K)$, we write $\K_v$ for the completion of
$\K$ with respect to $v$. If $v$ is archimedean, then $\K_v = \R$ or $\C$.
If $V \subset \cM(\K)$, we write $V[\fin]$ for the subset of finite places in $V$ and $V[\infty]$ for the archimedean ones.

Let $X_\K$ be a variety over $\K$. For every place $v$ of $\K$, define $X_v := X_\K \times_\K \spec \K_v$. Similarly, if $D$ is
an $\R$-divisor over $X_\K$ then we denote by $D_v$ its image under the base change. We write $X_v^{\an}$ for the
Berkovich analytification of $X_v$. We also define the global Berkovich analytification of $X_\K$ as
\begin{equation}
  X_\K^{\an} := \bigsqcup_v X_v^{\an}.
  \label{<+label+>}
\end{equation}
Comparing to \cite{yuanAdelicLineBundles2023}, this space is called the \emph{restricted analytic space} of $X_\K$ by Yuan
and Zhang. If $V$ is a set of places, we also define
\begin{equation}
  X_V^{\an} := \bigsqcup_{v \in V} X_v^{\an}.
  \label{<+label+>}
\end{equation}
In particular, we define
\begin{equation}
  X_\K^{\an} [\fin] := \bigsqcup_{v \in M(\K)[\fin]} X^{\an}_v, \quad X_\K^{\an}[\infty] := \bigsqcup_{v \in M(\K)[\infty]}
  X^{\an}_v
  \label{<+label+>}
\end{equation}

If $\sX_{\OO_\K}$ is a variety over $\OO_\K$, we write $\sX_v$ for the base change
\begin{equation}
  \sX_v := \sX_{\OO_\K} \times_{\OO_\K} \spec \OO_{\K_v}.
  \label{<+label+>}
\end{equation}
Similarly, if $\sD$ is an $\R$-divisor over $\sX_{\OO_\K}$, we denote by $\sD_v$ its image under the base change.

We make the following convention, if $\sX_{\OO_\K}$ is a variety over $\OO_\K$ and $V$ is a set of places of $\K$, we
write $X_\K = \sX_{\OO_\K} \times_{\OO_\K} \spec \K$ and
\begin{equation}
  \sX_{\OO_\K}^{\an, V} := X_\K^{\an,V}.
  \label{eq:<+label+>}
\end{equation}

\subsection{Over a finitely generated field}
\label{subsec:over-finit-gener}
Suppose that $F$ is a finitely generated field over $\Q$ of transcendence degree $d \geq 1$. Let $\K$ be the
algebraic closure of $\Q$ in $F$, then $\K$ is a number field. We describe the set of absolute values over $F$. Let
$\sB_{\OO_\K}$ be a projective model of $\spec \OO_\K$. That is a projective variety over $\OO_\K$ with function field
$F$. Every point in $\sB(\C)$ yields an archimedean place over $F$ and every irreducible component $\Gamma \subset \sB_\m =
\sB_{\OO_\K} \times \spec \kappa (\m)$ of every special fiber yields a non-archimedean absolute value. We write
$\cM_B (F)$ for the set of all places obtained with $\sB_{\OO_\K}$.  In what follows when we work with a finitely
generated field over $\Q$, we will set a model $\sB_{\OO_\K}$ once and for all and work with the places over that model.
In particular, like in the number field case we have a finite number of non-archimedean absolute values over a maximal
ideal of $\OO_\K$.

If $X_F$ is a quasiprojective variety over $F$, and $X_\K$ is a quasiprojective model of $X_F$ over $\K$. Then, we have
\begin{equation}
  X_F^{\an} = X_\K (\C) \cup \bigsqcup_{\Gamma \subset \sB_{\OO_\K}} X_F^{\an, \Gamma}
  \label{eq:<+label+>}
\end{equation}
and we have an embedding $X_F^{\an} \hookrightarrow X_\K^{\an}$ which is continuous,
injective, with a dense image by Lemma 3.1.1 of \cite{yuanAdelicLineBundles2023}.

\section{Local theory}\label{sec:local theory}
\subsection{Model arithmetic divisor and metrised line bundles}
Let $\K_v$ be a complete field with respect to a discrete absolute value $v$ and let $X$ be a projective variety over $\K_v$.
Let $D = \sum_i a_i D_i$ be an $\R$-Cartier divisor. A \emph{Green function} of $D$ is a continuous function
$g : X_{\K_v}^{\an} \setminus (\Supp D)^{\an} \rightarrow \R$ such that for every $q \in (\Supp D)^{\an}$ if $z_i$ is a
local equation of $D_i$ we have that
\begin{equation}
  g + \sum_i a_i \log \left| z_i \right|_v
  \label{<+label+>}
\end{equation}
extends to a continuous function at $q$. An \emph{arithmetic divisor} $\overline D$ over $X_{\K_v}$ is the data of an
$\R$-Cartier divisor $D$ over $X$ and of a Green function $g_{\overline D}$ of $D$.

A \emph{metrised} line bundle $\overline L$ over $X_{\K_v}$ is the data of a $\Q$-line bundle over $X_{\K_v}$ and a
metric on the space of sections of $L$ i.e for every $x \in X^{\an}_{\K_v}$ we have a real function $\left| \cdot
\right|_x$ on the stalk of $L$ at $x$ such that for every open subset $U \subset X_{\K_v}^{\an}$ and any $s \in H^0
(U, L)$ we have that the function
\begin{equation}
  x \in U \mapsto \left| s(x) \right|_x \in \R
  \label{<+label+>}
\end{equation}
is continuous.

Let $X_{\K_v}$ be a projective variety over $\K_v$ and let $D = \sum_i a_i D_i$ be an $\R$-Cartier divisor over $X$. A
\emph{model} of $(X,D)$ is the data of $(\sX_{\OO_v}, \sD)$ where $\sX_{\OO_v}$ is a projective variety over $\OO_v$
such that its generic fiber is $X_{\K_v}$ and $\sD = \sum_i a_i \sD_i$ is an $\R$-Cartier divisor over $\sX_{\OO_v}$
such that ${\sD_i}_{|X_{\K_v} = D_i}.$
Every model induces a Green function of $D$ over $X_{\K_v}^{\an}$ as follows. Consider the reduction map $r_{\sX_v} :
X^{\an} \rightarrow \sX_v \times_{\OO_v} \spec \kappa_v$, for any $x \in X^{\an} \setminus \Supp D$ let $z_i$ be a local
equation of $\sD_i$ at $r_{\sX_v}(x)$ then we define 
\begin{equation}
  g_{(\sX,\sD)} = - \sum_i a_i \log \left| z_i \right|.
  \label{<+label+>}
\end{equation}

Since we deal with $\R$-Cartier divisor there is some subtlety to consider here. Since $v$ is discretely valued, every
$\R$-Cartier divisor induces an $\R$-Weil divisor. The effectiveness of the $\R$-Weil divisor does not imply the
effectiveness of the $\R$-Cartier divisor in general but this is not important as we consider Green functions. In
particular we have the following lemma from \cite{abboudUnlikelyIntersectionsProblem2024}. 

\begin{lemme}\label{lemme:effectiveness-on-Green-function}
  Let $\sD$ be an $\R$-Cartier divisor over $\sX_v$. Then the induced $\R$-Weil divisor is effective if and only if
  $g_{\sX_v, \sD} \geq 0$. In particular, if $x \in X^{\an}_{\K_v}$ is such that $r_{\sX_v} (x) \not \in \Supp_W
  (\sD)$, then $g_{(\sX_v, \sD)} (x) = 0$ where $\Supp_W (\sD)$ is the support of the induced $\R$-Weil divisor.
\end{lemme}

\section{Global theory over a number field}
\label{sec:global-theory}
\subsection{Adelic divisors over a projective variety}
Let $\K$ be  a number field and $X_\K$ a projective variety over $\K$.

An adelic divisor $\overline D$ is the data of a divisor $D$ over $X_\K$ and a family
$(\overline D_v)_{v \in \cM(\K)}$ of arithmetic extensions of $D_v$ over $X_v$ such that
\begin{enumerate}
  \item If $v$ is archimedean and $\sigma : \K \hookrightarrow \K_v = \C$ is invariant by complex conjugation, then
    the Green function of $\overline D_v$ is also invariant by complex conjugation.
    \item There exists a model $(\sX_{\OO_\K}, \sD)$ of $(X_\K,D)$ over $\OO_\K$ and an open subset $V \subset \spec
      \OO_\K$ such that for every finite place $v \in V[\fin]$, $\overline D_v$ is the model arithmetic divisor induced
      by $\sD_v$.
  \end{enumerate}

Every rational function $P$ over $\K$ induces a model arithmetic divisor
\begin{equation}
  \hat \div (P) := (\div(P), (- \log |P|_v)_{v \in \cM(\K)}).
\end{equation}
Such divisors are called principal.

An adelic line bundle $\overline L$ is the data of a line bundle $L$ over $X_\K$ and a family of metrisations $(\overline
L_v)_{v \in \cM(\K)}$ of $L_v$ over $X_v$ such that there exists a model $(\sX_{\OO_\K}, \sL)$ of $(X_\K,L)$ over
$\OO_\K$ and an open subset $V \subset \spec \OO_\K$ such that for every finite place $v \in V[\fin], \overline L_v$ is
induced by the model $\sL_v$.

\begin{dfn}\label{dfn:adelic-divisors-and-line-bundles}
  \begin{itemize}
\item A \emph{model adelic divisor} $\overline D$ is an adelic divisor such that there exists a model $(\sX_{\OO_\K}, \sD)$ of
$(X_\K,D)$ such that for every finite place $v$, $\overline D_v$ is induced by $\sD_v$. In that case we write $\overline
D = \overline \sD$.
\item A \emph{model adelic line bundle} is defined similarly and written as $\overline \sL$.

\item An adelic divisor $\overline D$ (resp. adelic line bundle $\overline L$) is \emph{semipositive} if for every
$v$, $\overline D_v$ (resp. $\overline L_v$) is semipositive.
\item
An adelic divisor $\overline D$ (resp. adelic line bundle $\overline L$) is \emph{integrable} if for every
$v$, $\overline D_v$ (resp. $\overline L_v$) is integrable.

\item An adelic divisor $\overline D = (D,g)$ is \emph{effective} if $g \geq 0$ in particular this implies that
$D$ is an effective divisor. We write $\overline D \geq \overline D'$ if $\overline D - \overline D'$ is effective.
\item An adelic divisor $\overline D$ is \emph{strictly} effective if $\overline D \geq 0$ and $g[\infty] > 0$.
  \end{itemize}
\end{dfn}

\begin{ex}\label{ex:horizontal-divisors-number-field}
  Let $X_\K$ be a projective variety over $\K$ and let $D$ be an $\R$-divisor over $X_\K$. We can write $D$ as a
  $\R$-Weil divisor $D = \sum_i \lambda_i E_i$ where $E_i$ is an irreducible closed subvariety of codimension 1 in
  $X_\K$. Let $\sX_{\OO_\K}$ be a projective model of $X_\K$ over $\OO_{\K}$, since $\OO_\K$ is Noetherian, the closure
  $\overline E_i$ of $E_i$ in $\sX_{\OO_\K}$ is an irreducible codimension 1 closed subvariety of $\sX_{\OO_\K}$ and
  also a Cartier divisor. We still write $D$ for the horizontal divisor $\sum_i \lambda_i \overline E_i$ in $\sX_{\OO_\K}$.
\end{ex}

\begin{dfn}
  Let $H$ be a divisor such that $\OO_X (H)$ is globally generated and let $P_1, \dots, P_r$ be generators of $\Gamma
  (X, H)$, we call the Weil metric of $H$ with respect to $(P_1, \dots, P_r)$ the Green function of $H$ defined by
  \begin{equation}
    g(x) = \log^+ \max \left( \left| P_1 (x) \right|, \dots, \left|P_r (x)\right| \right).
    \label{eq:<+label+>}
  \end{equation}
  It yields semipositive extension of $H$. In particular, every ample divisor admits a strictly effective semipositive model.
\end{dfn}

\subsection{Global intersection number}
\label{subsec:glob-inters-numb}
If $X_\K$ is a projective variety over $\K$ of dimension $n$ and $\overline D_0, ..., \overline D_n$ are
integrable adelic divisors that intersect properly, then we define the global intersection number
\begin{equation}
  \label{eq:47}
  \overline D_0 \cdots \overline D_n = \sum_{v \in \cM(\K)} (\overline D_{0,v} \cdots \overline D_{n,v}).
\end{equation}
If any of the $\overline D_i$ is principal, then the global intersection number vanishes. Thus, we have a well
defined global intersection number for integrable adelic line bundles: if $\overline L_0, \dots, \overline L_n$ are
integrable line bundles over $X_\K$, let $s_i$ be a rational section of $L_i$ such that $\div (s_0), \dots,
\div(s_n)$ intersect properly, then
\begin{equation}
  \overline L_0 \dots \overline L_n := \hat \div(s_0) \dots \hat \div(s_n).
  \label{eq:<+label+>}
\end{equation}

\begin{prop}
  \label{prop:3}
  If every $\overline L_i$ is semipositive, then $\overline L_0 \cdots \overline L_n \geq 0$.
\end{prop}

We can define the height function $h_{\overline L}$ associated to an integrable adelic line $\overline L$. For
every closed $\overline \K$-subvariety $Z$ of $X$, one has
\begin{equation}
  \label{eq:48}
  h_{\overline L} (Z) := \frac{(\overline L_{|Z})^{\dim Z + 1}}{L_{|Z}^{\dim Z}}.
\end{equation}

In particular, if $\overline D$ is an adelic divisor and $p \in X_\K \setminus \Supp D$ is a closed point, then
\begin{equation}
  h_{\overline D} (p) = \sum_{v \in \cM (\K)} \frac{1}{\deg (p)}\sum_{q \in \Gal (\overline \K / \OO_\K) \cdot p} n_v
  g_{\overline D, v} (q)
  \label{eq:<+label+>}
\end{equation}
where $n_v = [\K_v : \Q_v]$.

\subsection{Over quasiprojective varieties}
The main reference for this section is \cite{yuanAdelicLineBundles2023}.
Let $U_\K$ be a normal quasiprojective variety over a number field $\K$. The construction is similar as in the local
setings.

Let $\sU_{\OO_\K}$ be a quasiprojective model of $U_\K$ over $\OO_\K$. A \emph{model adelic divisor} over $\sU_{\OO_\K}$
is a model adelic adelic divisor $\overline \sD$ over a projective model of
$\sU_{\OO_\K}$ over $\OO_\K$.
We write
$\hat \Div (\sX_{\OO_\K}, \sU_{\OO_\K})$ for the set of model adelic divisor
defined over a fixed projective model $\sX_{\OO_\K}$ of $\sU_{\OO_\K}$. Since the system of projective models of
$\sU_{\OO_\K}$ is a projective system, we can define the direct limit
\begin{equation}
  \hat \Div (\sU_{\OO_\K} / \OO_\K)_\mod := \varinjlim_{\sX_{\OO_\K}} \hat \Div(\sX_{\OO_\K}, \sU_{\OO_\K}).
  \label{<+label+>}
\end{equation}
Then, an adelic divisor over $U_\K$ is a Cauchy sequence of model adelic divisor with respect to a boundary divisor. It
comes with a \emph{Green function} which is the limit of the Green function of the divisor of the Cauchy sequence. We
write $\hat \Div(U_\K, \OO_\K)$ for the space of adelic divisors over $U_\K$. We also define a subset of adelic divisors
which is more suitable for our needs. We define $\hat \DivInf(U_\K/ \OO_\K)$ for the set of adelic divisors over
$U_\K$ \emph{suported at infinity}. That is if $\overline D = \lim \overline \sD_i$, then ${D_i}_{|U_\K} = 0$.

This is not restrictive as if $\overline D \in \hat \Div (U_\K)$, then
\begin{equation}
  \overline D \in \hat \DivInf ((U_\K \setminus \Supp D_{|U_\K}) / \OO_\K).
  \label{eq:not-restrictive}
\end{equation}

\begin{dfn}
An adelic divisor $\overline D$ over $U_\K$ is
\begin{itemize}
  \item \emph{vertical} if it is the limit of vertical model adelic divisors.
  \item \emph{strongly nef} if for the Cauchy sequence $(\overline \sD_i)$ defining
it we can take for every $\overline \sD_i$ a semipositive model adelic divisor.
\item \emph{nef} if there exists a strongly nef adelic divisor $\overline A$ such that for all $m \geq 1, \overline D +
  m \overline A$ is strongly nef.
\item \emph{integrable} if it is the difference of two strongly nef adelic divisors.
\end{itemize}
\end{dfn}

We define adelic line bundles over $U_\K$ similarly and write $\hat \Pic (U_\K, \OO_\K)$ for the set of adelic line
bundles over $U_\K$.
If $\overline L_0, \dots, \overline L_n$ are integrable, the global intersection number $\overline L_0, \cdots,
\overline L_n$ is defined as
\begin{equation}
  \label{eq:49}
  \overline L_0 \cdots \overline L_n := \lim_i \overline \sL_{0,i} \cdots \overline \sL_{n,i}.
\end{equation}
and the height function $h_{\overline L}$ is also well defined (see \cite{yuanAdelicLineBundles2023} \S 4.1 and \S 5.3).

We say that a nef adelic line bundle $\overline L$ is \emph{big} if $\overline L^{d+1} > 0$ (this implies that the
geometric intersection number $L^d$ is also $> r0)$).

\begin{rmq}\label{rmq:global-intersection-number}
  The definition of the global intersection number relies heavily on the fact that if $\overline \sL_0, \dots, \overline
  \sL_n$ are model semipositive adelic line bundles over a projective variety over $\K$, then $\overline \sL_0 \cdots
  \overline \sL_n \geq 0$.
\end{rmq}

As for the adelic divisors we define $\hat \Pic_\infty (\sU_{\OO_\K} / \OO_\K)_\mod$ for the set of model adelic line bundles
$\overline \sL$ such that $\sL_{|\sU_{\OO_\K}}$ is isomorphic to the trivial line bundle, we write $\hat \Pic_\infty
(\sU_{\OO_\K} / \OO_\K)$ for the completion of $\hat \Pic_\infty (\sU_{\OO_\K} / \OO_\K)_\mod$ with respect to the boundary topology and
we define
\begin{equation}
  \hat \Pic_\infty (U_\K / \OO_\K) := \varinjlim_{\sU_{\OO_\K}} \hat \Pic_\infty (\sU_{\OO_\K} / \OO_\K).
  \label{eq:<+label+>}
\end{equation}

\begin{prop}[\cite{yuanAdelicLineBundles2023} \S 2.5.5]\label{prop:functoriality}
  If $f : X_\K \rightarrow Y_\K$ is a morphism between quasiprojective varieties over $\K$, then there is a pullback operator
  \begin{equation}
    f^* : H(Y_\K) \rightarrow H(X_\K)
    \label{<+label+>}
  \end{equation}
  where $H = \hat \Div(\cdot), \hat \DivInf (\cdot / \OO_\K), \hat \Pic(\cdot / \OO_\K), \hat \Pic_\infty (\cdot /
  \OO_\K)$
  that preserves model, strongly nef, nef and integrable adelic divisors. If $g$ is the Green function of $\overline D
  \in \hat \DivInf(Y_\K / \OO_\K )$, then the Green function of $f^* \overline D$ is $g \circ f^{\an}$.
\end{prop}

\subsection{Chambert-Loir measures}\label{subsec:chambert-loir-measures}

If $\overline L_1, \cdots, \overline L_n$ are integrable adelic line bundles over $U_\K$, then for every
$v \in \cM(\K)$ we have the measure $c_1(\overline L_1) \cdots c_1(\overline L_n)_v$ defined as follows. Let
$\Omega \Subset U_\K^{\an, v}$  be a relatively compact open subset and for every $i, s_i$ is an invertible local section of
$L_{i|\Omega}$. Then, for every $j \geq 1$,
\begin{equation}
  \label{eq:65}
  \left(c_1(\overline L_{1,j}) \cdots c_1(\overline L_{n,j})_v\right)_{|\Omega} = dd^c (-\log ||s_1||_j) \wedge \cdots \wedge dd^c (- \log
  ||s_n||_j).
\end{equation}
Since the functions $\log ||s_i||_j$ converges uniformly over $\Omega$, we define
$\left(c_1(\overline L_1) \cdots c_1(\overline L_n)_v\right)_{|\Omega}$ as the weak limit of the measures
$\left(c_1(\overline L_{1,j}) \cdots c_1(\overline L_{n,j})_v\right)_{|\Omega}$. By \cite[Lemma 5.4.4]{yuanAdelicLineBundles2023}, we have that the total mass of the signed measure is
$c_1(\overline L_1) \cdots c_1 (\overline L_n)_v = c_1(L_1) \cdots c_1(L_n)$ for every place $v$.

\section{Over a finitely generated field}
\label{sec:over-fin-gen} Let $F$ be a finitely generated field over $\Q$. Set $d = \tr.deg(F/\Q)$.
Let $U_F$ be a quasiprojective variety over $F$. This setting will be called the \emph{arithmetic setting}. Let $\K$ be a
number field contained in $F$. An adelic divisor/line
bundle over $U_F$ is an adelic divisor/line bundle over any quasiprojective model $X_\K$ of $U_F$ over $\K$, more
precisely we define
\begin{equation}
  \label{eq:50}
  \hat \Div(U_F / \OO_\K) := \varinjlim_{X_\K} \hat \Div(X_\K / \OO_\K), \quad \hat \Pic (U_F / \OO_\K) := \varinjlim_{X_\K}
  \hat \Pic(X_\K / \OO_\K).
\end{equation}
 and
\begin{equation}
  \label{eq:50}
  \hat \DivInf(U_F / \OO_\K) := \varinjlim_{X_\K} \hat \DivInf(X_\K / \OO_\K), \quad \hat \Pic (U_F / \OO_\K) :=
  \varinjlim_{X_\K} \hat \Pic_\infty(X_\K / \OO_\K).
\end{equation}
The notions of semipositive, nef, integrable and vertical adelic divisors/ line
bundles follow from the direct limit. Of course if $d = 0$ and thus $F = \K$ we
recover the same definition as in the previous section.
We have the continuous embedding $U_F^{\an} \hookrightarrow X_\K^{\an}$, notice that the dimension of $X_\K$
over $\K$ is equal to $n + d$. If $\overline D$ is an adelic divisor over $U_F$ and $X_\K$ is a
quasiprojective model of $U_F$ where $\overline D$ is defined, then $\overline D$ defines a continuous
Green function
\begin{equation}
  \label{eq:51}
  g: (X_\K \setminus \Supp D_{|X_\K})^{\an} \rightarrow \R
\end{equation}
that restricts to a continuous function
\begin{equation}
  \label{eq:52}
  g: (U_F \setminus \Supp D_{|U_F})^{\an} \rightarrow \R.
\end{equation}
 If $\overline L_1, \cdots, \overline L_n$ are integrable adelic line bundles over $U_F$, then for every place
 $w \in (\spec F)^{\an}$ we have the measure
 \begin{equation}
   \label{eq:53}
   c_1(\overline L_1)_w \cdots c_1(\overline L_n)_w
 \end{equation}
 over $U_F^{\an, w}$.

 \subsection{Vector heights}\label{subsec:vector-heights}
 If $U_F$ is a quasiprojective variety over $F$ of dimension $n$ one expects to be able to define the
intersection number $\overline L_0 \cdots \overline L_n$ for integrable line bundles. This has been done by
Yuan and Zhang in \cite{yuanAdelicLineBundles2023} \S 4 using the Deligne pairing. We do not obtain a number
but an adelic line bundle over $\spec F$. Namely,
\begin{prop}\label{prop:deligne-pairing}
  There is a multilinear map
  \begin{equation}
    \overline L_0 \dots \overline L_n \in \hat \Pic(U_F / \OO_\K)^{n+1} \mapsto \langle \overline L_1 , \dots, \overline L_n
    \rangle \in \hat \Pic (F / \OO_\K)
    \label{eq:<+label+>}
  \end{equation}
  such that if every $\overline L_i$ is strongly nef (resp. nef, resp. integrable), then $\langle \overline L_1, \dots,
  \overline L_n \rangle$ also is strongly nef (resp. nef, resp. integrable). Furthermore, if $\overline H_1, \dots,
  \overline H_d \in \hat \Pic (F / \OO_\K)$ are integrable, then
  \begin{equation}
    \langle \overline L_0 , \dots , \overline L_n \rangle \cdot \overline H_1 \dots \overline H_d = \overline L_0 \cdots
    \overline L_n \cdot \pi^* \overline H_1 \dots \pi^* \overline H_d.
    \label{eq:<+label+>}
  \end{equation}
\end{prop}
This allows to define the vector height associated to $\overline L$. For every $p \in U_F (\overline F)$, we define
\begin{equation}
  \mathfrak h_{\overline L} (p) := \frac{1}{\deg (p)}\langle \overline L_{|p} \rangle \in \hat \Pic (F / \OO_\K)
  \label{eq:<+label+>}
\end{equation}
via the Deligne pairing induced by $ p \rightarrow \spec F$.

\subsection{Moriwaki heights}
\label{subsec:moriwaki-heights}
\begin{dfn}
  \label{dfn:1}
  An arithmetic polarisation of $F$ is the data of model semipositive adelic line bundles over $\spec F$ with ample
  underlying line bundles. We write $\overline H$ for the data of $\overline H_1, \dots, \overline H_d$.

  A model semipositive adelic line bundle $\overline H$ with ample underlying line bundle is said to satisfy the
  \emph{Moriwaki condition} if $\overline H^{d+1} = 0$ and $H^d > 0$.
\end{dfn}
Notice that contrary to \cite{yuanAdelicLineBundles2023}, we only take arithmetic polarisation which come from
\emph{model} adelic line bundles. This will be enough for our needs. In particular, for every non-archimedean place
$v$ of $\K$, the measure $c_1 (\overline H_1) \cdots c_1 (\overline H_{d-1})_v$ is a Dirac measure over $B_\K^{\an,v}$.

Let $B_\K$ be a projective of $\spec F$ over $\K$ such that every $\overline H_i$ is defined over
$B_\K$. By Lemma \ref{lemme:system-quasiprojective-models}, we can assume that
\begin{equation}
  \label{eq:55}
  \hat \Div (U_F / \OO_\K) = \varinjlim_{X_\K} \hat \Div(X_\K / \OO_\K), \quad \hat \Pic(U_F / \OO_\K) = \varinjlim_{X_\K}
  \hat \Pic(X_\K / \OO_\K)
\end{equation}
where for all the quasiprojective models $X_\K$ of $U_F$ we have a morphism $\pi: X_\K \rightarrow B_\K$. Indeed, let
$X_\K \rightarrow Y_\K$ be a quasiprojective model of $U_F \rightarrow \spec F$. Then, the birational map
$Y_\K \dashrightarrow B_\K$ induces a morphism $Y_\K' \hookrightarrow B_\K$ where $Y_\K '$ is an open subset of
$Y_\K$ and we can replace $X_\K$ by the preimage of $Y_\K'$ by Lemma \ref{lemme:system-quasiprojective-models}.

For integrable line bundles $\overline L_0, \cdots, \overline L_n$ we define the intersection number
\begin{equation}
  \label{eq:56}
  \left( \overline L_0 \cdots \overline L_n \right)_{\overline H} := \overline L_0 \cdots
  \overline L_n \cdot \pi^* \overline H_1 \cdots \pi^* \overline H_d.
\end{equation}

The Moriwaki height of a closed $\overline F$-subvariety with respect to an integrable adelic line bundle
$\overline L$ and to the polarisation $\overline H = (\overline H_1, \cdots, \overline H_d)$ is
\begin{equation}
  \label{eq:57}
  h_{\overline L}^{\overline H} = \frac{\left(\overline L_{|Z}^{\dim Z + 1}\right)_{\overline H}}{L_{|Z}^{\dim Z}}.
\end{equation}

\begin{thm}[Theorem 5.3.1 of \cite{yuanAdelicLineBundles2023}, Northcott property]\label{thm:northcott-property}
  Let $F$ be a finitely generated field over $\Q$. Let $X_F$ be a \emph{projective} variety over
  $F$ and $\overline L$ an integrable adelic line bundle with ample underlying line bundle. For any arithmetic
  polarisation $\overline H_1, \dots, \overline H_d$ of $F$, with each $\overline H_i$ big and for any $A, M > 0$ the set
  \begin{equation}
    \left\{ x \in X_F (\overline F): \deg (x) \leq A, h_{\overline L}^{\overline H} (x) \leq M  \right\}
    \label{eq:<+label+>}
  \end{equation}
  is finite.
\end{thm}

The Moriwaki height of a model adelic divisor $(\overline D,g)$ is of the form
\begin{equation}
  \forall p \not \in \Supp (D) (\overline F), \quad h_{\overline D}^{\overline H} (p) = \sum_{q \in \Gal(\overline F /
  F)(p)}\int_{B(\C)} g(q(b)) c_1 (\overline H_1) \cdots c_1 (\overline H_d) (b) + \sum_{q \in \Gal (\overline F / F)}
  \sum_{\Gamma \subset \sB_{\OO_\K}} g_\Gamma (q) \left( \overline H_1 \cdots \overline H_d\right)_{|\Gamma}.
  \label{eq:<+label+>}
\end{equation}
Comparing with the number field case, we have a sum over non-archimedean places which is similar as in the number field
case, the main difference is for the Archimedean places. There are infinitely many of them so we integrate over all of
them with respect to a measure given by the polarisation.

\begin{prop}\label{prop:moriwaki-height-as-integral}
  Let $U_F$ be a quasiprojective variety over $F$ and $\overline D \in \hat \DivInf(U_F / \OO_\K)$, then we have the formula
  \begin{equation}
    \forall p \in U_F(\overline F), \quad h_{\overline D}^{\overline H} (p) = \sum_{q \in \Gal(\overline F /
  F)(p)}\int_{B(\C)} g(q(b)) c_1 (\overline H_1) \cdots c_1 (\overline H_d) (b) + \sum_{q \in \Gal (\overline F / F)}
  \sum_{\Gamma \subset \sB_{\OO_\K}} g_\Gamma (q) \left( \overline H_1 \cdots \overline H_d\right)_{|\Gamma}.
    \label{eq:<+label+>}
  \end{equation}
\end{prop}
\begin{proof}
  Let $\overline D_i$ be a Cauchy sequence converging to $\overline D$. The sum $h_{\overline D_i}(p)$ is $\leq
  h_{\overline \sD_0} (p)$ for some boundary divisor $\overline \sD_0$. Thus, for the non-archimedean places this is
  just the fact that the convergence $g_{\overline D_i} \rightarrow g_{\overline D}$ is locally uniform. Now for the
  archimedean place. We have by \cite{chenArithmeticIntersectionTheory2021}, that the function
  \begin{equation}
    b \in B(\C \setminus \overline \K) \mapsto g_{\overline D_i} (p(b))
    \label{eq:<+label+>}
  \end{equation}
  is integrable with respect to $c_1 (\overline H_1) \dots c_1 (\overline H_d)$ and they are all bounded by the
  measurable function $b \mapsto g_{\overline \sD_0 (p(b))}$. Since $g_{\overline D_i} (p(b)) \rightarrow g_{\overline
  D}(p(b))$, by The Lebesgue dominated convergence theorem, we have that
  \begin{equation}
    \lim_i \int_{B (\C)} g_{\overline D_i} (p (b)) = \int_{B(\C)} g_{\overline D} (p(b)).
    \label{eq:<+label+>}
  \end{equation}
  and that shows the result.
\end{proof}

\subsection{Arithmetic equidistribution}\label{subsec:arithm-equidistribution}
We have the following result from \cite{yuanAdelicLineBundles2023} that we restate in a version more suitable for our
needs. 
\begin{thm}[{\cite[Theorem 5.4.5]{yuanAdelicLineBundles2023}}]\label{thm:arithm-equid}
  Let $U_F$ be a quasiprojective variety over a field $F$ and $\overline L$ be a nef adelic line bundle such that
  $\deg_L U_F = L^{\dim U_F} > 0$. Let $X_\K$ be a quasiprojective model of $U_F$ over $\K$ such that $\overline L \in \hat
  \Pic(X_\K / \OO_\K)$ and let
  $\overline H$ be an arithmetic polarisation of $F$ satisfying the Moriwaki condition defined over $\sB$. If $(x_k)_{k
  \geq 0}$ is a generic sequence of $X(\overline F)$ such that $h^{\overline H}_{\overline L} (x_k) \rightarrow h^{\overline
H}_{\overline L} (U_F)$, then
\begin{enumerate}
  \item for every $\Gamma \subset \sB$, and for every compactly supported continuous function $\phi : U_F^{\an, \Gamma}
    \rightarrow \R$  we have
    \begin{equation}
      \lim_m \frac{1}{\deg(x_m)} \sum_{y \in \Gal (\overline F / F) \cdot x_m} \phi(y) = \int_{U_F^{\an, \Gamma}} \phi
      c_1 (\overline L)^d_\Gamma.
      \label{eq:<+label+>}
    \end{equation}
  \item If $v$ is an archimedean place of $\K$ (i.e an embedding $\K \hookrightarrow \C$),
    then or every continous function with compact support $X_\K (\C) \rightarrow \R$ we have
    \begin{equation}
      \lim_m \frac{1}{\deg(x_m)} \sum_{y \in \Gal(\overline F / F) \cdot x_m } \int_{\sB(\C)} \phi(y(b)) c_1 (\overline
      H)^d = \int_{\sB(\C)} \left( \int_{X^{\an,b}} \phi_b c_1 (\overline L)_b^d \right) d \mu_{\overline H} (b).
      \label{eq:<+label+>}
    \end{equation}
  \end{enumerate}
\end{thm}
Theorem F of \cite{chenHilbertSamuelFormulaEquidistribution} suggests that equidistribution should hold even without the
Morkiwaki condition.

\subsection{Moriwaki condition}\label{subsec:moriwaki-condition}
Some results about volume estimate and equidistribution
require the Morkiwaki condition but it not compatible with the requirements for the Northcott property. To get rid of
the Moriwaki condition we will need the following result
\begin{prop}[Lemma 3.2 of \cite{yuanArithmeticHodgeIndex2021}]\label{prop:moriwaki-condition}
  Let $\overline L \in \hat \Pic(F / \OO_\K)$ is a nef adelic line bundle be such that for every $\overline H \in \hat
  \Pic(F / \OO_\K)_\mod$
  satisfying the Moriwaki condition we have
  \begin{equation}
    \overline L \cdot \overline H^d = 0
    \label{eq:<+label+>}
  \end{equation}
  then $\overline L$ is numerically trivial in $\hat \Pic(F / \OO_\K)$.
\end{prop}

\section{Geometric setting}\label{sec:geometric-setting}
Let $\K$ be any field and let $F$ be a finitely generated field over $\K$ such that $\tr.deg F/\K = d \geq 1$. We equip $\K$
with the trivial absolute value. We will call this situation the \emph{geometric setting}. In contrast with the
arithmetic setting, we only deal with non-archimedean places in the
geometric case. We even restrict to a specific kind of places: A
\emph{geometric polarisation} of $F$ is the data of a projective model $B$ of $F$ over $\K$ which is regular in codimension 1 and
nef line bundles $H_1, \dots, H_{d-1}$ over $B$. We
define the set of \emph{$B$-places} $\cM_B (F)$ of $F$ as the set of points of codimension 1 in $B$. A codimension 1 point in
$B$ is a generic point $\eta_E$ for a prime divisor $E$ of $B$. The order of vanishing $\ord_E$ along $E$ defines a
seminorm over $F$:
\begin{equation}
  \left| f \right|_E := e^{- \ord_E (f)}.
  \label{eq:<+label+>}
\end{equation}
We write $\cM(F) = \bigcup_{B} \cM_B (F)$ for the set of places of $F$. If $E$ is a prime divisor in $B$ then the local
ring $\OO_{\eta_E}$ at its generic point is a discrete valuation ring because $B$ is regular in codimension 1. Write
$F_E$ for the completion of $F$ with respect to $\ord_E$, then the valuation ring of $F_E$ is $\hat{\OO_{\eta_E}}$ the
completion of $\OO_{\eta_E}$ with respect to $\ord_E$.

If $V \subset B$ is an open subset and $E \in \cM_B (F)$ a place of $F$, we say that $E$ \emph{lies} above $V$ if
$\eta_E \in V$.

In particular, if $\tr.deg F/\K = 1$, then there exists a unique projective curve $B_k$ over $\K$ such that $\K(B) = F$ and
$\cM (F) = \cM_B (F)$.

\subsection{Adelic divisors and line bundles}\label{subsec:adelic-divisors-and-line-bundles-function-field}
The definition of adelic line bundles is easier in this setting because there are no archimedean places. The arithmetic
intersection number is just the intersection number of Chern classes of line bundles over projective varieties over $\K$.

Let $B$ be a projective model of $F$ over $\K$ and $U_F$ a quasiprojective variety over $F$. Let $X_B$ be a
quasiprojective model of $U_F$ over $B$ (notice that
$\dim_\K X_B = \dim_F U_F + d$). A \emph{model adelic divisor} over $X_B$ is the data $(\sX_B, \sD)$ of a projective
model $\sX_B$ of $X_B$ over $B$ and a divisor $\sD$ over $\sX_B$. By the valuative criterion of
properness we have a reduction map
\begin{equation}
  r_{\sX_B} : \sX_B^{\an} := \bigsqcup_{E \in \cM_B(F)} \sX_B^{\an, E} \rightarrow \sX_B.
  \label{eq:<+label+>}
\end{equation}
And using this reduction map, the divisor $\sD$ defines for every $E \in \cM_B (F)$ a Green function of $\sD_E$ over
$X_B^{\an, E}$. We write $\hat \Div (X_B / B)_{\mod}$ for the set of model adelic divisors over $X_B$.

An adelic divisor over $X_B$ is a Cauchy sequence of model adelic divisors with respect to a boundary divisor $\sD_0$ of
$X_B$. Here the notation $\sD \geq 0$ means that the divisor $\sD$ is effective and we write $\hat \Div (X_B/ B)$ for
the set of adelic divisors over $X_B$. An adelic divisor over $U_F$ is then an
adelic divisor over any quasiprojective model $X_B$ of $U_F$ over $B$ for any projective model $B$ of $F$ over $\K$. In
other terms we define the set of adelic divisors over $U_F$ as
\begin{equation}
  \hat \Div (U_F/ \K) := \varinjlim_{B} \varinjlim_{X_B} \hat \Div (X_B / B).
  \label{eq:<+label+>}
\end{equation}

Analogously, an adelic line bundle over $U_F$ is a Cauchy sequence of model adelic line bundles and we write $\hat \Pic
(U_F/ \K)$ for the set of adelic line bundles over $U_F$. If $B, H_1, \cdots,
H_{d-1}$ is a geometric polarisation of $F$, then we have the intersection number
\begin{equation}
  (L_0 \cdots L_n)_{\overline H} = L_0 \cdots L_n \cdot H_1 \cdots H_{d-1}.
  \label{eq:<+label+>}
\end{equation}

We also define $\hat \Pic_\infty (U_F / \K)$ and $\hat \DivInf (U_F / \K)$.
\begin{rmq}
  When $F$ is a finitely generated field over a number field $\K$ we can use either $\hat \Div (U_F/ \OO_\K)$ or
  $\hat \Div (U_F/ \K)$ whether we want to work in the arithmetic setting or the geometric setting. In this paper, we
  will always use the arithmetic setting because of the Northcott property.
\end{rmq}

\subsection{Moriwaki heights}\label{subsec:moriwaki-heights-function-field}
Let $L$ be an adelic line bundle over $U_F$, let $Z \subset U_F$ be a closed $ \overline F$-subvariety, we define the height
\begin{equation}
  h_{\overline L}^{ \overline H} = \frac{( L_{|Z})^{\dim Z +1}_{\overline H}}{(\dim Z +1) \deg_{L} (Z)}.
  \label{eq:<+label+>}
\end{equation}
In particular,
\begin{equation}
  \forall x \in U_F( \overline F), \quad h_{ L}^{ \overline H}(x) = \sum_{E \in \cM_B(F)} \sum_{y \in \Gal(\overline F /
  F)\cdot x} \ord_E (s(y)) \left( H_1 \dots  H_{d-1} \cdot E \right)
  \label{eq:<+label+>}
\end{equation}
where $s$ is any non vanishing section of $L$ at the image of $x$ in $U_F$.

Suppose $d> 1$. Let $ H$ be a nef adelic line bundle over $B$, we say that $ H$ satisfies the \emph{Moriwaki condition} if
$ H^d = 0$ and if there exists a curve $C$ over $\K$ with a morphism $B \rightarrow C$ with function field $K \subset F$
of transcendence degree 1 over $\K$ such that over the generic fiber $H_K^{d-1} > 0$. If $d=1$ we make the convention
that the Moriwaki condition is always satisfied.
Theorem \ref{thm:arithm-equid} also holds in this setting but there are no archimedean places.
We also have the Northcott property in the geometric setting
\begin{thm}[Theorem 5.3.1 of \cite{yuanAdelicLineBundles2023}, Northcott
  property]\label{thm:northcott-property-geometric}
  Let $F$ be a finitely generated field over a \emph{finite} field $\K$. Let $X_F$ be a \emph{projective} variety over
  $F$ and $L \in \hat \Pic (X_F/ \K)$ an integrable adelic line bundle with ample underlying line bundle. For
  any nef and big geometric polarisation $\overline H_1, \dots, \overline H_d$ of $F$, for any $A, M > 0$ the set
  \begin{equation}
    \left\{ x \in X_F (\overline F): \deg (x) \leq A, h_{ L}^{\overline H} (x) \leq M  \right\}
    \label{eq:<+label+>}
  \end{equation}
  is finite.
\end{thm}

We conclude this section with the following result
\begin{prop}\label{prop:intersection-not-trivial}
  Let $E \in \cM (F)$, then there exists a geometric polarisation $H$ of $F$ satisfying the Moriwaki condition such
  that $E \cdot H^{d-1} > 0$
\end{prop}
\begin{proof}
  Let $B_\K$ be a projective model of $F$ over $\K$ such that $E \subset B_\K$ we can suppose up to blowing up that there
  exists a morphism
  \begin{equation}
    q : B_\K \rightarrow (\P^1_\K)^d
    \label{eq:<+label+>}
  \end{equation}
  Let $\pi_i : (\P^1_\K)^d \rightarrow \P^1$ be the projection to the $i$-th coordinate. Let $H_i = \pi_i^* \OO(1)$. Set
  \begin{equation}
    H = q^* \left( H_1 + \dots + H_{d-1} \right),
  \end{equation}
  $H$ satisfies the Moriwaki condition since
  \begin{equation}
    \left( H_1 + \dots + H_{d-1} \right)^d = 0
    \label{eq:<+label+>}
  \end{equation}
  and if $K = \pi_d^* \K(\P^1_\K)$, then
  \begin{equation}
    H_K^{d-1} = \left( H_1 + \dots + H_{d-1} \right)^{d-1} \cdot H_d = 1 > 0.
    \label{eq:<+label+>}
  \end{equation}
  If $H \cdot E = 0$, then there must exist $1 \leq i \leq d-1$ such that $E$ is an irreducible component of a fiber
  of the morphism $\pi_i \circ q$. Thus, for a generic choice of $q$ this is not the case.
\end{proof}

\section{The algebraic torus}\label{sec:algebraic-torus}
\subsection{Proof of Theorem \ref{bigthm:rigidity-periodic-points} for the algebraic
torus}\label{subsec:rigidity-algebraic-torus}

Let $\K$ be any algebraically closed field, Any $\K$-automorphism of $\G_m^2$ is of the form
\begin{equation}
  f(x,y) = \left( \alpha x^a y^b, \beta x^c y^d \right)
  \label{eq:<+label+>}
\end{equation}
where $\alpha, \beta \in \K^\times$ and $\begin{pmatrix}
  a & b \\ c & d
\end{pmatrix}
\in \GL_2 (\Z)$. Let $f_A$ be the automorphism induced by a matrix $A \in \GL_2(\Z)$, then the set of periodic points of
$f_A$ is
\begin{equation}
\Per(f_A) = \mathbb U \times \mathbb U
  \label{eq:periodic-points-fA}
\end{equation}
where $\mathbb U$ is the set of roots of unity. We can always conjugate an automorphism $f$ by a translation to an
automorphism of the form $f_A$ (this bounds to finding a fixed point of $f$ and conjugating by a translation such that
the fixed point is $(1,1)$). Theorem \ref{bigthm:rigidity-periodic-points} is easier for the algebraic torus as we have
the following result.
\begin{prop}
  Let $f,g \in \Aut(\G_m^2)$ be two loxodromic automorphisms, then \begin{equation}
    \Per(f) \cap \Per(g) \neq \emptyset \Leftrightarrow \Per(f) = \Per(g).
    \label{eq:<+label+>}
  \end{equation}
\end{prop}
\begin{proof}
  Suppose $\Per(f) \cap \Per(g)$ is not empty. Up to iterating $f$ and $g$, we can suppose that $f,g$ have a common
  fixed point. Therefore up to conjugation they are of the form $f_A$ and  $g_B$ with $A,B \in \GL_2(\Z)$ and thus they
  have the same set of periodic points by \eqref{eq:periodic-points-fA}.
\end{proof}

\subsection{A characterisation of the algebraic torus}\label{subsec:charac-algebraic-torus}
Let $\K$ be an algebraically closed field of any characteristic.
A \emph{quasi-abelian} variety is an algebraic group $Q$ such that there exists an exact sequence of algebraic groups
\begin{equation}
  1 \rightarrow T \rightarrow Q \rightarrow A \rightarrow 1
  \label{eq:<+label+>}
\end{equation}
where $T$ is an algebraic torus and $A$ is an abelian variety. For any algebraic variety $V$, there exists a universal
quasi-abelian variety $\QAlb(V)$ equipped with a morphism $q : V \rightarrow \QAlb(V)$ such that any morphism $V
\rightarrow Q$ where $Q$ is a quasi-abelian variety factors through $q$. We call $\QAlb(V)$ the \emph{quasi-Albanese}
variety of $V$. If $V$ is projective, then $\QAlb(V)$ is the Albanese variety of $V$. For a general reference for
quasi-abelian varieties, we refer to \cite{serreExposesSeminaires195019992001, fujinoQUASIALBANESEMAPS2015}.

We have the following characterisation of the algebraic torus. It was proven in \cite{abboudDynamicsEndomorphismsAffine2023}
\S 10.
\begin{thm}
  Let $X_0$ be a normal affine surface over an algebraically closed field. If $X_0$ admits a loxodromic automorphism,
  then either $\QAlb(X_0) = 0$ or $X_0 \simeq \G_m^2$.
\end{thm}
\section{Picard-Manin space at infinity}
For this section, let $S_F$ be a normal affine surface over a field $F$.
\subsection{Completions}
A completion of $S_F$ is a projective model $X_F$ of $S_F$ over $F$. We call
$X_F \setminus \iota_{X_F} (S_F)$ the \emph{boundary} of $S_F$ in $X_F$. By \cite{goodmanAffineOpenSubsets1969} Proposition
1, it is a curve. We will also refer to it at the part "at infinity" in $X_F$. For any completion $X_F$ of $S_F$ we define
$\DivInf (X_F)_\A = \oplus \A E_i$ where $\A = \Z, \Q, \R$ and $X_F \setminus S_F = \bigcup E_i$, the space of
$\A$-divisors at infinity. For any
two completions $X_F,Y_F$ we have a birational map $\pi_{{X_F Y_F}} = \iota_{Y_F} \circ \iota_{X_F}^{-1} : X_F \dashrightarrow Y_F$. If this map is
regular, we say that $\pi_{{X_F Y_F}}$ is a \emph{morphism of completions} and that $X_F$ is \emph{above} $Y_F$. For any
completion
$X_F,Y_F$ there exists a completion $Z$ above $X_F$ and $Y_F$. Indeed, take $Z$ to be a resolution of indeterminacies of
$\pi_{{X_F Y_F}} : X_F \dashrightarrow Y_F$. A morphism of completions defines a pullback and a pushforward operator
$\pi_{{X_F Y_F}}^*, (\pi_{{X_F Y_F}})_*$ on divisors and Néron-Severi classes. We have the projection formula,
\begin{equation}
  \forall \alpha \in \NS (X_F), \beta \in \NS (Y_F), \alpha \cdot \pi_{{X_F Y_F}}^* \beta = (\pi_{{X_F Y_F}})_* \alpha
  \cdot \beta.
  \label{<+label+>}
\end{equation}

\begin{lemme}\label{lemme:divisors-at-infinity-in-Neron-Severi}
  Let $S_F$ be a normal affine surface with $\QAlb(S_F) = 0$, then for every completion $X_F$ of $S_F$, the natural group
  homomorphism
  \begin{equation}
    \DivInf(S_F)_\A \rightarrow \NS(X_F)_\A
    \label{eq:<+label+>}
  \end{equation}
  is injective.
\end{lemme}
\begin{proof}
  First of all, since $\QAlb(S_F) = 0$, we have that $\OO(S_F)^\times = F^\times$, thus the group homomorphism
  $\DivInf(S_F)_\A \rightarrow \Pic (X_F)_\A$ is injective. Then, the group homomorphism $\Pic (X_F) \rightarrow \NS
  (X_F)$ is injective because its kernel is the dual of the Albanese variety of $X_F$ which must be trivial since
  $\QAlb(S_F) = 0$.
\end{proof}

\subsection{Weil and Cartier classes}
If $\pi_{YX} : Y_F \rightarrow X_F$ are two completions of $S_F$ then we have the embedding defined by the
pullback operator
\begin{equation}
  \pi_{YX}^* : \DivInf(X_F)_\A \hookrightarrow \DivInf(Y_F)_\A.
  \label{<+label+>}
\end{equation}
We define the space of Cartier divisors at infinity of $S_F$ to be the direct limit
\begin{equation}
  \Cinf (S_F) := \varinjlim_{X_F} \DivInf(X_F)_\R.
  \label{<+label+>}
\end{equation}
In the same way we define the space of Cartier classes of $S_F$
\begin{equation}
  \cNS (S_F) := \varinjlim_{X_F} \NS(X_F)_\R.
  \label{<+label+>}
\end{equation}
An element of $\cNS (S_F)$ is an equivalence class of pairs $(X_F, \alpha)$ where $X_F$ is a completion of $S_F$ and
$\alpha \in \NS (X_F)_\R$ such that $(X_F, \alpha) \simeq (Y_F, \beta)$ if and only if there exists a completion $Z$ above
$X_F,Y_F$ such that $\pi_{ZX}^* \alpha = \pi^*_{ZY} \beta$. We say that $\alpha \in \cNS (S_F)$ is \emph{defined} in $X_F$ if it
is represented by $(X_F, \alpha)$.
We have a natural embedding $\Cinf (S_F) \hookrightarrow \cNS(S_F)$, we still write $\Cinf (S_F)$ for its image in $\cNS
(S_F)$. We also define the space of Weil classes
\begin{equation}
  \wNS (S_F) := \varprojlim_{X_F} \NS(X_F)_\R
  \label{<+label+>}
\end{equation}
where the compatibility morphisms are given by the pushforward morphisms $(\pi_{YX})_* : \NS (Y_F) \rightarrow \NS(X_F)$ for
a morphism of
completions $\pi_{YX}: Y_F \rightarrow X_F$.
An element of this inverse limit is a family $\alpha = (\alpha_{X_F})_{X_F}$ such that if $X_F,Y_F$ are two completions of $S_F$
with $Y_F$ above $X_F$, then $(\pi_{YX})_* \alpha_{Y_F} = \alpha_{X_F}$. We call $\alpha_{X_F}$ the \emph{incarnation} of
$\alpha$ in $X_F$. We have a natural embedding $\cNS (S_F) \hookrightarrow \wNS (S_F)$. We also define the space of Weil
divisors at infinity
\begin{equation}
  \Winf (S_F) := \varprojlim_{X_F} \DivInf(X_F)_\R
  \label{<+label+>}
\end{equation}
and we have the commutating diagram
\begin{equation}
  \begin{tikzcd}
    \Cinf(S_F) \ar[r, hook] \ar[d, hook] & \cNS (S_F) \ar[d, hook] \\
    \Winf(S_F) \ar[r, hook] & \wNS(S_F)
  \end{tikzcd}
  \label{<+label+>}
\end{equation}
Thanks to the projection formula, the intersection form defines a perfect pairing
\begin{equation}
  \cNS (S_F) \times \wNS (S_F) \rightarrow \R
  \label{<+label+>}
\end{equation}
defined as follows. If $\alpha \in \cNS (S_F)$ is defined in $X_F$ and $\beta \in \wNS (S_F)$, then
\begin{equation}
  \alpha \cdot \beta = \alpha_{X_F} \cdot \beta_{X_F}
  \label{<+label+>}
\end{equation}
An element $\alpha \in \Winf(S_F)$ is \emph{effective} if for every completion $X_F$, $\alpha_{X_F}$ is an effective
divisor. We write $\alpha \geq \beta$ if $\alpha - \beta$ is effective. An element $\beta \in \wNS (S_F)$ is
\emph{nef} if for every completion $X_F, \beta_{X_F}$ is nef.

\subsection{The Picard-Manin space of $S_F$}
We provide $\wNS (S_F)$ with the topology of the inverse limit, we call it the weak topology, $\cNS(S_F)$ is
dense in $\wNS (S_F)$ for this
topology. Analogously, $\Cinf (S_F)$ is dense in $\Winf(S_F)$.

We define $\cD_\infty$ for the set of prime divisors at infinity. An element of $\cD_\infty$ is an equivalence class of
pairs $(X_F,E)$ where $X_F$ is a completion of $S_F$ and $E$ is a prime divisor at infinity. Two pairs $(X_F,E), (Y_F, E')$
are equivalent if the birational map $\pi_{{X_F Y_F}}$ sends $E$ to $E'$. We will just write $E \in \cD_\infty$ instead of
$(X_F,E)$. We define the function $\ord_E : \Winf (S_F) \rightarrow \R$ as follows. Let $\alpha \in S_F$, if $X_F$ is any
completion where $E$ is defined (in particular $(X_F,E)$ represents $E \in \cD_\infty$), then $\alpha_{X_F}$ is of the form
\begin{equation}
  \alpha_{X_F} = a_E E + \sum_{F \neq E} a F
  \label{<+label+>}
\end{equation}
and we set $\ord_E (\alpha_{X_F}) = a_E$. This does not depend on the choice of $(X_F,E)$.
\begin{lemme}[\cite{boucksomDegreeGrowthMeromorphic2008} Lemma 1.5]\label{lemme:weak-topology}
  The map
  \begin{equation}
    \alpha \in \Winf (S_F) \mapsto (\ord_E (\alpha))_{E \in \cD_\infty} \in \R^{\cD_\infty}
    \label{<+label+>}
  \end{equation}
  is a homeomorphism for the product topology.
\end{lemme}

\begin{rmq}
  In \cite{boucksomDegreeGrowthMeromorphic2008} or \cite{cantatNormalSubgroupsCremona2013}, the Picard-Manin space is
  defined by allowing blow up with arbitrary centers not only at infinity. Since we study dynamics of automorphism of
  $S_F$ the indeterminacy points are only at infinity. This justifies our restricted definition of the Picard-Manin
  space. A similar construction is used in \cite{favreDynamicalCompactificationsMathbf2011} for the affine plane.
\end{rmq}

\subsection{Spectral property of the dynamical degree}
If $f \in \Aut (S_F)$ we define the operator $f^*$ on $\cNS(S_F)$ as follows. Let $\alpha \in \cNS(S_F)$
defined in a completion $X_F$. Let $Y_F$ be a completion of $S_F$ such that the lift $F: Y_F \rightarrow X_F$ of $f$ is
regular. We define $f^* \alpha$ as the Cartier class defined by $F^* \alpha$. This does not depend on the choice of $X_F$
or $Y_F$. We write $f_*$ for $(f^{-1})^*$. If $X_F$ is a completion of $S_F$, we write $f_X^* : \DivInf(S_F) \rightarrow
\DivInf(S_F)$ for the following operator:
\begin{equation}
  f_X^* (D) = (f^* D)_X
  \label{<+label+>}
\end{equation}
where we consider the class of $D$ and $f^*D$ in $\Cinf (S_F)$. We also define the operator $f_X^* : \NS(X_F) \rightarrow
\NS(X_F)$ in a similar way.
\begin{prop}[Proposition 2.3 and Theorem 3.2 of \cite{boucksomDegreeGrowthMeromorphic2008}]
  The operator $f^*$ extends to a continuous operator $f^* : \wNS (S_F) \rightarrow
  \wNS (S_F)$.
\end{prop}
If $\lambda_1 (f) > 1$, then $\lambda_1$ is simple and there is a spectral gap property.
\begin{thm}[Theorem 3.5 of \cite{boucksomDegreeGrowthMeromorphic2008} and Theorem 3.28 of
  \cite{abboudDynamicsEndomorphismsAffine2023}]\label{thm:dynamic-automorphism-picard-manin}
  Let $f$ be a loxodromic automorphism of $S_F$, there exist nef elements $\theta^+, \theta^- \in \overline \Cinf
  (S_F)$ unique up to renormalisation such that
  \begin{enumerate}
  \item $\theta^+$ and $\theta^-$ are effective.
  \item $(\theta^+)^2 = (\theta^-)^2 = 0, \theta^+ \cdot \theta^- = 1$.
  \item $f^* \theta^+ = \lambda_1 \theta^+, (f^{-1})^* \theta^- = \lambda_1 \theta^-$
    \label{<+label+>}
  \end{enumerate}
\end{thm}

\subsection{Compatibility with adelic divisors}\label{sub:sec-compatibility-adelic-divisor}
Let $\K$ be either a number field or any field over which $F$ is finitely generated. Let $B_\K$ be either $\OO_\K$ in
the arithmetic setting or $\K$ in the geometric setting. Recall that in the definition of
$\hat \Div_\infty (S_F / B_\K)$, we impose that if $\overline D \in \hat \Div_\infty(S_F / B_\K)$, then $D_{|S_F} = 0$.
We have a forgetful group homomorphism
\begin{equation}
c: \hat \Div_\infty(S_F / B_\K)_\mod \rightarrow \Winf(S_F)
  \label{<+label+>}
\end{equation}
defined as follows. Let $\cU$ be a quasiprojective model of $S_F$ over $B_\K$ and let $\overline \sD$ be a model adelic
divisor on $\cU$. Then, $c (\overline \sD) = \sD_{F}$ is the restriction of the horizontal part of $\sD$ to $S_F$.

\begin{prop}
  The group homomorphism $c$ extends to a continuous group homomorphism
  \begin{equation}
    c : \hat \Div_\infty(S_F / B_\K ) \rightarrow \Winf (S_F).
    \label{<+label+>}
  \end{equation}
  Furthermore, if $\overline D$ is integrable then $c (\overline D) \in \overline \Cinf (S_F)$.
\end{prop}
\begin{proof}
  Let $\overline D \in \hat \Div_\infty (S_F /B_\K)$ be given by a Cauchy sequence of model adelic divisors $(\overline
  \sD_i)$. Let $X_F$ be a completion of $S_F$. There exists a sequence $\epsilon_i$ converging to zero such that
  \begin{equation}
  - \epsilon_i \overline \sD_0 \leq \overline \sD_j - \overline \sD_i \leq \epsilon_i \overline \sD_0
    \label{<+label+>}
  \end{equation}
  Applying $c$, we get (write $D_j = c(\overline \sD_j)$)
  \begin{equation}
    -\epsilon_i D_0 \leq D_j - D_i \leq \epsilon_i D_0
    \label{<+label+>}
  \end{equation}
  Thus, for every $E \in \sD_\infty, \ord_E (D_i)$ is a Cauchy sequence and converges to a number $\ord_E (D)$. By Lemma
  \ref{lemme:weak-topology} this defines a Weil divisor $c (D) \in \Winf (S_F)$. It is clear that $c$ is continuous, again
  using Lemma \ref{lemme:weak-topology}.

  If $\overline D$ is integrable, then it is the difference of two strongly nef adelic divisors and nef classes in $\Winf
  (S_F)$ belong to $\Cinf (S_F)$, therefore $c(\overline D) \in \overline \Cinf$.
\end{proof}
We will drop the notation $c(\overline D)$ and just write $D = c(\overline D)$.

\section{Dynamics of a loxodromic automorphism}
Here we state some main results from \cite{abboudDynamicsEndomorphismsAffine2023} where dynamics of endomorphisms of
affine surfaces was studied using valuative techniques. Regarding loxodromic automorphisms of affine varieties the main
results from loc. cit. is that a loxodromic automorphism on a normal affine surface has a Henon-like dynamics at
infinity. More precisely, we state the following result.
\begin{prop}[Theorem 14.18 and Theorem 14.4 of
  \cite{abboudDynamicsEndomorphismsAffine2023}]\label{prop:dynamics-loxodromic-automorphism}
  Let $F$ be a field and $S_F$ a normal affine surface over $F$ with $\QAlb(S_F) = 0$. Let $f \in \Aut (S_F)$ be a loxodromic
  automorphism. There exists a completion $X_F$ of $S_F$ and closed points $p_{-}, p_{+} \in (X_F \setminus S_F)(F)$ such that
  \begin{enumerate}
      \item $p_+ \neq p_-$
      \item There exists $N_0 \geq 1$ such that for all $N \geq N_0, f^{\pm N}$ contracts $X_F \setminus S_F$ to $p_\pm$.
    \item $f^\pm$ is defined at $p_{\pm}, f^{\pm 1} (p_\pm) = p_\pm$ and $p_\mp$ is the unique
      indeterminacy point of $f^{\pm N}$ for $N$ large enough.
  \item If $\lambda_{1}(f) \not \in \Z$, there exists local algebraic coordinates $(z,w)$ at $p_\pm$ such that $zw = 0$
    is a local equation of the boundary and $f^{\pm 1}$ is of the form
    \begin{equation}
      f^{\pm 1} (z,w) = \left( z^a w^b \phi, z^c w^d \psi \right)
      \label{eq:forme-normale-monomiale}
    \end{equation}
    with $ad - bc = \pm 1$ and $\phi, \psi$ invertible.
    \item If $\lambda_{1}(f) \in \Z$, then there exists local coordinates $(z,w)$ at $p_{\pm}$ such that $z = 0$ is a
      local equation of the boundary and $f^{\pm}$ is of the form
          \begin{equation}
            f^{\pm}(z,w) = (z^{a} \phi, w^{c}z^d\psi).
            \label{eq:forme-normale-henon}
          \end{equation}
          with $a \geq 2, c,d \geq 1$, $\phi$ invertible, $\psi$ regular.
  \item $f$ is algebraically stable over $X_F$ and $f_X^* \theta_X^+ = \lambda_1 \theta_X^+, (f_X^{-1})^* \theta_X^- =
    \lambda_1 \theta_X^-$.
  \item $\Supp \theta_X^\pm = X_F \setminus S_F$.
\end{enumerate}
    Furthermore, the subset of completions of $S_F$ satisfying all these properties is cofinal in the set of
    completions of $S_F$.
\end{prop}

\begin{cor}[Corollary 3.4 from \cite{cantatBersHenonPainleve2009}]\label{cor:no-invariant-curves}
  Suppose $\QAlb (S_F) = 0$, then any loxodromic automorphism of $S_F$ does not admit any invariant algebraic curves.
\end{cor}
\begin{proof}
  Let $f \in \Aut(S_F)$ be loxodromic. Let $X_F$ be a completion of $S_F$ given by Proposition
  \ref{prop:dynamics-loxodromic-automorphism}.
  If $C \subset S_F$ was an invariant algebraic curve, then its closure $\overline C$ in $X_F$ should intersect
  $X_F \setminus S_F$. Since the boundary is contracted by $f$, we must have $p_+ \in \overline C$ and $f : \overline C
  \rightarrow C$ is an automorphism with a superattractive fixed point. This is a contradiction.
\end{proof}

\begin{cor}\label{cor:periodic-points-are-algebraic}
  Let $S_F$ be a normal affine surface over a field $F$ and let $f \in \Aut(S_F)$ be a loxodromic automorphism. For any field
  extension $F \hookrightarrow L$,
  \begin{equation}
    \Per (f_L) = \Per(f)
    \label{eq:<+label+>}
  \end{equation}
  i.e every periodic point of $f$ is defined over the algebraic closure of $F$.
\end{cor}
\begin{proof}
  If $S_F = \G_m^2$, then up to translation by an $\overline F$-point $\Per (f) = \mathbb U \times \mathbb U$ and any
  periodic point of $f$ is defined over the algebraic closure of the prime field of $F$ so the result is trivial.

  Suppose that $\QAlb(S_F) = 0$ and that there exists $p \in \Per(f_L)$ not defined over $\overline F$ (in
    particular, $L$ is not algebraic over $F$). We write $f_L$ for the base change of $f$ over $L$. We can suppose up to
    replacing $f$ by an iterate that $p$ is a fixed point. Then, the Galois orbit $\Gal (L / \overline F \cap L) \cdot
    p$ of $p$ defines an infinite number of fixed point of $f_L$ over $S_F \times_{\spec F} \spec L$, its Zariski
    closure in $S_F \times_{\spec F} \spec L$ is either of dimension 1 or 2. In both cases this contradicts the fact
    that $f_L$ is a loxodromic automorphism, either because $f_L \neq \id$ or by Corollary
    \ref{cor:no-invariant-curves}.
\end{proof}

\begin{prop}\label{prop:operateur-pull-back}
  Let $X_F$ be a completion of $S_F$ given by Proposition \ref{prop:dynamics-loxodromic-automorphism}, replace $f$ by one of
  its iterates such that $f^{\pm 1}$ contracts $X_F \setminus S_F$ to $p_\pm$. Then,
  \begin{enumerate}
    \item For all $R \in \DivInf({X_F})_\R$ such that $p_{+} \not \in \Supp R, f_{X_F}^* R = 0$ and $\theta^- \cdot R = 0$.
    \item If $\lambda_{1}(f) \in \Z$, then
      $\left\{\theta_{{X_F}}^{+} \right\} \cup  \left\{E : p_{+} \not \in \Supp E \right\}$ is a basis of $\DivInf({X_F})_{\R}$.
    \item If $\lambda_{1}(f) \not \in \Z$, there exists $D^- \in \DivInf({X_F})_\R$
          such that
          \begin{enumerate}
            \item $f_{X_F}^* D^- = \frac{1}{\lambda_1} D^-$
            \item $D^{-} \cdot \theta^{-} = 0$
            \item $\{\theta_{X_F}^+, D^-\} \cup \left\{ E : p_{+} \not \in \Supp E \right\}$
                  is a basis of $\DivInf({X_F})$
          \end{enumerate}
  \end{enumerate}
\end{prop}
\begin{proof}
  If $E$ is a prime divisor at infinity such that $p_{+} \not \in \Supp E$ then
  $f_{X_F}^* E = 0$ because every prime divisor at infinity is contracted to $p_+$ by
  $f$. Now if $R$ satisfies $f_{X_F}^* R = 0$, then
  \begin{equation}
    0 = f_{X_F}^* R \cdot \theta_{X_F}^- = R \cdot (f_{X_F}^{-1})^* \theta_{X_F}^- = \lambda_1 R \cdot \theta_{X_F}^-.
    \label{<+label+>}
  \end{equation}
  Thus $R \cdot \theta_{X_F}^- = 0$. This shows (1).

  If $\lambda_{1}(f) \in \Z$, then by Proposition
\ref{prop:dynamics-loxodromic-automorphism}, the family
$(\theta_{{X_F}}^{+}, E : p_{+} \not \in \Supp E)$ has length equal to
$\dim \DivInf({X_F})_{\R}$. So we just need to show that it is a free family.
Suppose there exists $t \in \R$ and $R \in \DivInf({X_F})$ such that
$p_{+} \not \in \Supp R$ that satisfy
\begin{equation}
  \label{eq:2}
  t \theta_{{X_F}}^{+} + R = 0.
\end{equation}
Applying $f_{{X_F}}^{*}$ to \eqref{eq:2}, we get $t = 0$. Thus, $R = 0$ and we get the result.

  If $\lambda_{1}(f) \not \in \Z$, then by Proposition \ref{prop:dynamics-loxodromic-automorphism} $p_{+} = E_{+} \cap F_{+}$ where $E_{+}, F_{+}$ are two prime divisors at infinity. Since $f_{X_F}^+ \theta_{X_F}^+ = \lambda_1 \theta_{X_F}^+$, we have that
  \begin{equation}
    \theta_{X_F}^+ = \alpha E_+ + \beta F_+ + \cdots
    \label{<+label+>}
  \end{equation}
  where $(\alpha, \beta)$ is an eigenvector of $A = \begin{pmatrix}
    a & c \\
    b & d
  \end{pmatrix}$ of eigenvalue $\lambda_1$. Now, the other eigenvalue of $A$ is $\frac{1}{\lambda_1}$ by Proposition
  \ref{prop:dynamics-loxodromic-automorphism} (4) (up to replacing $f$ by $f^2$), let
  $(\gamma, \delta)$ be an associated eigenvector, then
  \begin{equation}
    f_{X_F}^* (\gamma E_+ + \delta F_+) = \frac{1}{\lambda_1} (\gamma E_+ + \delta F_+) + R
    \label{<+label+>}
  \end{equation}
  where $R$ is a divisor at infinity which support does not contain $E_+$ or $F_+$. Set $D^- = \gamma E_+ + \delta F_+ +
  \lambda_1 R$, then by (1), $D^-$ satisfies $f_{X_F}^* D^- = \frac{1}{\lambda_1} D^-$. This shows (3)(a).

  Now,
  \begin{equation}
    \frac{1}{\lambda_1} D^- \cdot \theta^- = \frac{1}{\lambda_1} D^- \cdot \theta_{X_F}^- = (f_{X_F}^* D^-) \cdot \theta_{X_F}^- =
    D^- \cdot (f_{{X_F}})_* \theta_{X_F}^- = \lambda_1 D^- \cdot \theta^-.
    \label{<+label+>}
  \end{equation}
  Thus $D^- \cdot \theta^- = 0$. This shows (3)(b).

  Finally, we just have to show that the family $\left\{ \theta_{X_F}^+, D^- \right\} \cup \left\{ E; E \not \in \left\{
  E_+, F_+ \right\} \right\}$ is free. Suppose that
  \begin{equation}
    \alpha \theta_{X_F}^+ + \beta D^- + R = 0
    \label{-eq4}
  \end{equation}
  with $\alpha, \beta \in \R$ and $E_+, F_+ \not \in \Supp R$. Intersecting with $\theta^-$ in \eqref{-eq4} and using (2)
  and (3), we get $\alpha = 0$. Then, applying $f_{X_F}^*$ to \eqref{-eq4} we get $\beta = 0$. Thus $R = 0$ and we have shown
  (3)(c).
\end{proof}

\section{An invariant adelic divisor}\label{sec:invariant-adelic-divisor}
In this section, we use an iterative process to construct an invariant adelic divisor for $f$ and $f^{-1}$. This process
is comparable to the construction of Green functions for Henon maps over the affine plane in
\cite{bedfordPolynomialDiffeomorphismsC21991, kawaguchiLocalGlobalCanonical2009, ingramCanonicalHeightsHenon2014}.
However, our approach follows the construction of the canonical height for polarised endomorphisms of projective
varieties. We adapt the strategy of \emph{Tate's limiting argument} (see for example \S 2 of
\cite{zhangSmallPointsAdelic1993}). In this section we consider only affine surfaces $S_F$ with $\QAlb(S_F) = 0$.

\begin{thm}\label{thm:invariant-adelic-divisor}
  Let $F$ be either a finitely generated field over a number field $\K$ (arithmetic setting) or a finitely generated
  field over any field $\K$ with $\tr.deg F / \K \geq 1$ (geometric setting) and set either $k = \OO_\K$ or $\K$. Let
  $S_F$ be a normal affine surface over $F$ and $f$ be a loxodromic automorphism of $S_F$, there
  exist two unique, up to normalisation, adelic divisors $\overline{\theta^+}, \overline{\theta^{-}} \in \hat
  \Div_\infty(S_F / k)$ such that
  \begin{equation}
    f^* \overline{\theta^+} = \lambda_1 \overline{\theta^+}, \quad (f^{-1})^{*} \overline \theta^{-} = \lambda_{1}
    \overline \theta^{-}.
    \label{<+label+>}
  \end{equation}
  Furthermore, $\overline{\theta^+}$ and $\overline{\theta^-}$ are strongly nef adelic divisors.
\end{thm}
\begin{rmq}
  With the notations of \S \ref{sub:sec-compatibility-adelic-divisor},
 we must have $c(\overline \theta^\pm) = \theta^\pm$ (up to multiplication by a positive constant) because of Theorem
 \ref{thm:dynamic-automorphism-picard-manin}. So our notation of $\overline \theta^\pm$ is compatible with Theorem
 \ref{thm:dynamic-automorphism-picard-manin}. Notice that since we work in the global setting, the models we consider
 are all Noetherian and we can use $\R$-divisors. This is crucial because in general for any completion $X_F$ of $S_F$,
 the divisors $\theta_{X_F}^{\pm}$ are $\R$-divisors.
\end{rmq}
This theorem was proven when $S_F = \cM_D$ is the Markov surface with algebraic parameter $D$ in
\cite{abboudUnlikelyIntersectionsProblem2024}. If $F$ is a number field, then the proof is analogous to the one in loc.
cit. If $\tr.deg F / \Q \geq 1$ or $\car F > 0$, the essence of the proof is the same as in the number field case but
more technical.

From now on, $k$ will either denote $\OO_\K$ in the arithmetic setting and $\K$ in the geometric setting and $\sB_k$ will
either denote a projective model of $f$ over $\spec \OO_\K$ in the arithmetic setting or over $\K$ in the geometric setting.
Start with the following lemma.
\begin{lemme}\label{lemme:convergence-diviseur-verticaux}
  For any vertical model adelic divisor $\overline M \in \hat \Div (S_F / k)$, we have
  \begin{equation}
    \frac{1}{\lambda^n} (f^n)^* \overline M \xrightarrow[n \rightarrow +\infty]{} 0
    \label{eq:<+label+>}
  \end{equation}
  in $\hat \Div (S_F / k)$.
\end{lemme}
\begin{proof}
  Let $q: \sX_k \rightarrow \sB_k$ be a projective model of $S_F \rightarrow \spec F$ over $k$ such that $\overline M$ is
  defined over $\sX_k$. Then, there exists an open subset $\cT_k \subset \sB_k$ such that $\sX_\cT := \sX_k
  \times_{\sB_k} \cT_k$ is flat over $\cT$ and $M$ is supported outside $\sX_\cT$. We can blow up $\sB_k \setminus \cT_k$
  and $\sX_k \setminus \sX_\cT$ such that $\cT$ admits a boundary divisor $\overline \sE_\cT$ in $\sB_k$. Then, $q^*
  \overline \sE_\cT$ is a boundary divisor of $\sX_\cT$ in $\sX_k$ and there exists $A > 0$ such that
  \begin{equation}
    - A q^* \overline \sE_\cT \leq \overline M \leq A q^* \overline \sE_\cT.
    \label{eq:<+label+>}
  \end{equation}
  And for every $N \geq 0$ we have $(f^N)^* q^* \overline \sE_\cT = q^* \overline \sE_\cT$ since $f$ induces the identity
  over $\sB_k$. Thus
  \begin{equation}
    -\frac{A}{\lambda^N} q^* \overline \sE_\cT \leq \frac{1}{\lambda^N} (f^N)^* \overline M \leq \frac{A}{\lambda^N} q^*
    \overline \sE_\cT.
    \label{eq:<+label+>}
  \end{equation}
\end{proof}

We show the following
\begin{prop}\label{prop:convergence-diviseur-model}
  Let $S_F$ be a normal affine surface over $F$ and let $f$ be a loxodromic automorphism of $S_F$, then there exists a
  unique $\overline{\theta^+} \in \hat \DivInf(S_F / k)$ such that for every model
  adelic divisor $\overline \sD \in \hat \Div_\infty(S_F / k)_{mod}$, one has
  \begin{equation}
    \label{eq:3}
    \frac{1}{\lambda_{1}^{N}} (f^{N})^{*} \overline \sD \rightarrow (\theta^- \cdot D)\overline \theta^{+}.
  \end{equation}
\end{prop}

The proof of this proposition will take the whole section.
Start with a completion $X_F$ of $S_F$ that satisfies Proposition \ref{prop:dynamics-loxodromic-automorphism}. We denote
by $s = f^{N_0}$ an iterate of $f$ such that $s^{\pm 1}$ contracts $X_F \setminus S_F$ to $p_\pm$ and such that if the
normal form of $s^{\pm 1}$ at $p_\pm$ is of the form \eqref{eq:forme-normale-henon}, then $c \geq 2$. We prove Proposition
\ref{prop:convergence-diviseur-model} for $s$ and then will deduce the result for $f$. We replace $\lambda_1$ by
$\lambda_1^{N_0}$.

\begin{rmq}\label{rmq:dynamical-compactification-model}
  If $\sX_k$ is a projective model of $X_F$ over $k$, then we can always assume that
  \begin{equation}
    \sZ_k = \overline{ \left\{ p_+\right\}}
    \cap \overline{ \left\{ p_- \right\}} = \emptyset. \label{eq:empty-intersection}
  \end{equation}
 Indeed, otherwise $\sZ_k$ is a closed subvariety of codimension
  $\geq 2$ in $\sX_k$ and does not intersect $X_F$. So the blow-up of $\sX_k$ along $\sZ_k$ is a projective model of
  $X_F$ over $k$ that satisfies \eqref{eq:empty-intersection}. We will always make this assumption from now on.
\end{rmq}

\begin{lemme}\label{lemme:good-model}
  Let $\sX_k$ be a projective model of $X_F$ over $\sB_k$ there exists an open subset $\cT_k \subset \sB_k$ such that
  $\sX_\cT := \sX_k \times_{\sB_k} \cT_k$ satisfies
  \begin{enumerate}
    \item \label{item:good-model1} $\sX_\cT$ is a projective model of $X_F$ over $\cT$.
    \item \label{item:empty-intersection-of-closure} The closure of $p_+$ and $p_-$ in $\sX_\cT$ have no point in common.
    \item \label{item:normal-form} If $(z,w)$ are local coordinates in $X_F$ at $p_\pm$ such that $s^{\pm 1}$ has
      local normal form \eqref{eq:forme-normale-monomiale} or \eqref{eq:forme-normale-henon} with regular functions
      $\phi, \psi$, then there exists an open neighbourhood $O^\pm$ of $\overline{\left\{ p_\pm \right\}}$ in $\sX_\cT$
      such that
      \begin{enumerate}
        \item \label{item:boundary-no-intersection-with-p} $\sY_\cT^\pm := \sX_\cT \setminus \OO^\pm$ is horizontal and
          its closure in $\sX_k$ does not intersect
          $\overline{ \left\{ p_\pm \right\}}$.
        \item \label{item:divisor-phi-psi-no-intersection-with-p} $z , w, \phi, \psi$ are regular functions over $O^\pm$
          and the horizontal components of $\div(\phi),
          \div(\psi)$ not intersecting $O^\pm$ do not intersect $\overline{ \left\{ p_\pm \right\}}$.
        \item \label{item:local-equation-of-p} $(z,w)$ are generators of the ideal sheaf of $\overline{ \left\{ p_\pm
          \right\}}$ over $O^\pm$.
      \end{enumerate}
    \item \label{item:good-model2} The indeterminacy locus of the birational map $s^{\pm 1} : \sX_\cT \dashrightarrow
      \sX_\cT$ is the closure of $p_\mp$ in $\sX_\cT$.
      \item \label{item:good-model3} If $\sU_\cT = \sX_\cT \setminus \overline{\partial_{X_F} S_F}$, then $s$ extends to an
        automorphism of $\sU_\cT$.
  \end{enumerate}
   \end{lemme}
   \begin{proof}
      We treat the case where the local normal form is monomial. Let $(z,w)$ be local coordinates at $p_+$ in $X_F$,
      such that we have
     \begin{equation}
       (s^{\pm 1})^* (z, w) = \left( z^a w^b \phi, z^c w^d \psi \right).
       \label{eq:<+label+>}
     \end{equation}
     then, $z,w, \phi, \psi$ induces rational functions over $\sX_k$. Let $O^+$ be the complement in $\sX_k$ of
     the union of
     \begin{enumerate}
       \item The vertical components of $\supp \div (\alpha)$ for $\alpha = z,w, \phi,\psi$.
       \item $\overline E$ for any irreducible component $E \subset \partial_{X_F} S_F$ such that $p_+ \not \in E$.
       \item The horizontal components of $\supp \div (\alpha)$ for $\alpha = z,w, \phi, \psi$ where there is a pole.
     \end{enumerate}
     Then, $z,w, \phi, \psi$ are regular functions over $O^+$ and $z = w = 0$ is an equation of $\overline p_+$ in
     $O^+$.

     We do the same procedure with $p_-$ which yields an open subset $O^-$. Now, let $\sZ$ in $\sX_k$ be the
     vertical closed subset defined as the union of
     \begin{enumerate}
       \item The vertical irreducible components of the complement of $O^+$ and $O^-$.
       \item The vertical irreducible components of $\Ind : s^{\pm 1} \sX_k \dashrightarrow \sX_k$.
       \item The vertical components of the exceptional locus of $s^{\pm 1}: \sX_k \dashrightarrow \sX_k$.
     \end{enumerate}
     We define $\cT = \sB_k \setminus q (\sZ)$ and we replace $O^\pm$ by $O^\pm \cap \sX_\cT$. Every condition in the
     lemma is satisfied except maybe for \ref{item:boundary-no-intersection-with-p} and
     \ref{item:divisor-phi-psi-no-intersection-with-p}. To ensure these two conditions hold, we
     blow up first $\overline{\sY_\cT^\pm} \cap \overline{ \left\{ p_\pm \right\}}$. This is a vertical blow-up and the
     center is outside $\sX_T$. Then, we can blow up $\overline{ \left\{ p_\pm \right\}} \cap E$ where $E$ runs through
     the horizontal components of $\div(\phi)$ and $\div (\psi)$ not intersecting in $O^\pm$. Again these are all
     vertical blow-ups.
   \end{proof}

   \begin{dfn}\label{dfn:def-of-V}
   In the arithmetic case, we write $V \subset \spec \OO_\K$ for the open subset which is the image of $\cT_{\OO_\K}$
   in $\spec \OO_\K$. In the geometric case, we set $V = \cT$.
   \end{dfn}

 \begin{ex}\label{ex:good-model}
   Let $F = \Q(t), S_F = \A^2_{\Q(t)}$ and
   \begin{equation}
   s(x,y) = \left( y, x + \frac{1}{2t} y^3 \right)
     \label{eq:<+label+>}
   \end{equation}
   Then, $X_F = \P^2_{\Q(t)}$ is a completion of $S_F$ that satisfies Proposition
   \ref{prop:dynamics-loxodromic-automorphism} with
   \begin{equation}
     p_- = [1:0:0], p_+ = [0:1:0].
     \label{eq:<+label+>}
   \end{equation}
   Then, $\sX_\Z = \P_\Z^2 \times \P_\Z^1$ is a projective model of $S_F$ over $\Z$ and we have the commutative diagram
   \begin{center}
   \begin{tikzcd}
     S_F = \A^2_{\Q(t)} \ar[r, hook] \ar[d] & \sX_\Z = \P_\Z^2 \times \P^1_\Z \ar[d, "q"] \\
     \spec \Q(t) \ar[r, hook] & \P^1_\Z = B_\Z
   \end{tikzcd}
   \end{center}
   We write $[X:Y:Z]$ for the projective coordinates over $\P^2_\Z$ and $[T:S]$ for the projective coordinates over
   $\P^1_\Z$. The rational map $s$ becomes
   \begin{equation}
     s \left( [X:Y:Z] , [T:S] \right) = \left( [2TYZ^2: 2TXZ^2 + SY^3 : 2TZ^3] , [S:T] \right).
     \label{eq:<+label+>}
   \end{equation}
   The indeterminacy locus of $s : \sX_\Z \dashrightarrow \sX_\Z$ is
   \begin{equation}
     \overline{ \left\{ p_\mp \right\}} \cup \left\{ T = Y = 0  \right\} \cup \left\{ 2 = S = 0 \right\} \cup \left\{
     2 = Y = 0 \right\} \cup \left\{ Z = S = 0 \right\}
   \end{equation}
   and the exceptional locus of $s$ is
   \begin{equation}
     \left\{ T = 0 \right\} \cup \left\{ 2 = 0 \right\}.
     \label{eq:<+label+>}
   \end{equation}
   We can do similar computations for $s^-1$ and we can show that the open subset
   \begin{equation}
     \cT_\Z = \P^1_{\Z} \setminus \left( \left\{ T = 0 \right\} \cup \left\{ S = 0 \right\} \cup \left\{ 2 = 0 \right\}
     \right).
     \label{eq:<+label+>}
   \end{equation}
   and $V = \spec \Z \setminus \left\{ (2) \right\}$.
 \end{ex}

\begin{rmq}\label{rmq:not-compact}
  It is important to notice that in general for any place $v$, $\sX_\cT^{\an,v}$ is not compact. Indeed,
  $\sX_k^{\an,v}$ is compact but we have removed $(q^{\an})^{-1} (\sB_k \setminus \cT_k)^{\an,v}$. For example, in Example
  \ref{ex:good-model}, we have for the archimedean place of $\Q$
  \begin{equation}
    \sX_\cT^{\an,v} = \P^2 (\C) \times \left(\P^1 (\C) \setminus \left\{ 0, \infty \right\}\right)
    \label{eq:<+label+>}
  \end{equation}
  where $0 = [0:1]$ and $\infty = [1:0]$.
\end{rmq}

\begin{lemme}
  \label{lemme:generic-places}
  Let $D \in \DivInf(X_F)_\R$ such that $s_X^* D = \mu D$ for some $\mu \in \R$ and let $\overline \sD$ be a model
  adelic extension of $D$. It is defined over a
  projective model $\sX_{k}$ of $X_F$ over $k$ and we can suppose that there exists a regular morphism $\sX_k
  \rightarrow \sB_k$ between projective varieties over $k$ with generic fiber $X_F \rightarrow \spec F$. Let $\cT_k
  \subset \sB_k$ is an open subset that satisfies Lemma
  \ref{lemme:good-model} and let $V$ be its associated set from Definition \ref{dfn:def-of-V}. Suppose no vertical component
  of the $\R$-Weil divisor $\sD$ is above $\cT_k$. For every finite place $v$ above $V$, we define the following open
  neighbourhood $U^-_v$ of $p_\pm$ in $\sX_\cT^{\an,v}$:
  \begin{equation}
    U^-_v := \left\{ x \in \sX_\cT^{\an ,v} :  r_{\sX_v} (x) = r_{\sX_v}(p_-)\right\}.
    \label{eq:<+label+>}
  \end{equation}
  Then, $s^{-1}$ is defined over $U^-_{v}$, $U^-_v$ is $s^{- 1}$-invariant and if $W_v^- = \sX_\cT^{\an,v}
  \setminus U^\pm_v$, then $W_v^-$ is $s$-invariant and
    \begin{equation}
      \left(g_{\left(\sX_v, \sD_v\right)} \circ s^{\an}\right)_{|W_v^-} = \mu {g_{\left(\sX_v, \sD\right)}}_{|W_v^-}.
      \label{eq:egalite-green-function-plus} 
    \end{equation}
  \end{lemme}
  \begin{proof}
    First, recall with Remark \ref{rmq:dynamical-compactification-model} that we always assume $\overline{\{p_+\}} \cap
    \overline{\{p_-\}} = \emptyset$. Thus, for every finite place $v$, $r_{\sX_v}(p_-) \neq r_{\sX_v}(p_+)$.
    By our assumption, $\Ind (s^{-1} : \sX_\cT \dashrightarrow \sX_\cT) = \overline{ \left\{ p_+ \right\}} \subset
    \sX_\cT$, therefore the indeterminacy locus of $s^{-1} : \sX_k \dashrightarrow \sX_k$ is equal to the union to the
    closure of $p_+$ in $\sX_k$ and some vertical components that are not above $\cT_k$.
    In particular, $s^{-1}$ induces an endomorphism $\sX_\cT^{-1} \setminus \overline{\{ p_+\}} \rightarrow \sX_\cT \setminus
    \overline{\{ p_+\}}$ and $s^{-1} (\sX_\cT \cap \overline{\{ p_-\}}) = \sX_\cT \cap \overline{\{ p_-\}}$. The same relations
    holds over the Berkovich analytification. Thus, for any finite place $v$, $U^-_v$ is $s^{-1}$-invariant.

  We show \eqref{eq:egalite-green-function-plus}. 
  Let $\pi :Y_F \rightarrow X_F$ be a minimal sequence of blow-ups such that $s$ lifts to a regular map $S : Y_F
  \rightarrow X_F$. The morphism $\pi$ is a composition of point blow-ups above $p_-$. Let $\pi : \sY_k \rightarrow
  \sX_k$ be the induced blow-ups of $\sX_k$, $\sY_k$ is then a projective model of $Y_F$ over $k$ and $S$ lifts to a
  birational map $S : \sY_k \dashrightarrow \sX_k$. The indeterminacy locus of $S$ is a vertical subvariety which does not
  lie above $\sT$. We blow it up and still call $\sY_k$ the obtained projective variety over $k$ with a birational
  morphism $S : \sY_k \rightarrow \sX_k$.

  We have an isomorphism $\pi : \pi^{-1}\left( \sX_\cT \setminus \overline{\{p_-\}}\right) \xrightarrow{\sim} \sX_\cT \setminus
  \overline{\{p_-\}}$ because the center of $\pi$ is outside $\sX_\cT \setminus \overline{\{p_-\}}$. So for any
  finite place $v \in V$ we also have the isomorphism
\begin{equation}
  \label{eq:4}
  \pi_{v}^{\an} : (\pi_{v}^{-1})^{\an} (W^-_{v}) \xrightarrow{\sim} W^-_{v}
\end{equation}
because
\begin{equation}
  \pi_v \circ r_{\sY_v} = r_{\sX_v} \circ \pi_v^{\an}.
  \label{eq:<+label+>}
\end{equation}
Now, $s_{X_F}^* D = \mu D$, therefore the vertical part of the $\R$-Weil divisor 
\begin{equation}
  S^* \sD - \mu \pi^* \sD
  \label{<+label+>}
\end{equation}
has no support over $\cT$ and its horizontal part is $\pi$-exceptional, i.e supported over $p_-$ on the generic fiber.
This implies that
\begin{equation}
  \Supp (S^* \sD - \mu \pi^* \sD) \cap r_{\sY_v} (W_v^-) = \emptyset.
  \label{<+label+>}
\end{equation}
This yields \eqref{eq:egalite-green-function-plus} by Lemma \ref{lemme:effectiveness-on-Green-function}
\end{proof}

From this lemma, we show
\begin{prop}\label{prop:convergence-diviseur-propre-model}
  Let $X_F$ be a completion of $S_F$ that satisfy Proposition \ref{prop:dynamics-loxodromic-automorphism} and $D \in
  \DivInf(X_F)_\R$ such that $f_{X_F}^* D = \mu D$. If $(\sX_{k}, \overline \sD)$ is a model adelic extension of
  $(X_F,D)$, then the sequence
\begin{equation}
  \label{eq:5}
 \frac{1}{\lambda_{1}^{N}} (s^{N})^{*} \overline \sD
\end{equation}
converges to
\begin{enumerate}
  \item zero if $|\mu| < \lambda_{1}$,
  \item to $\overline \theta^{+} (X_F)$ if $\mu = \lambda_{1}$ and $D = \theta_{X_F}^{+}$ where $\overline
    \theta^{+}(X_F)$ is an adelic divisor over $S_F$ which a priori depends on $X_F$.
\end{enumerate}
 \end{prop}
 We split the proof in two parts. We can suppose that there exists a projective model $\sB_k$ of $\spec F$ over $k$ and
 a morphism $\sX_k \rightarrow \sB_k$ with generic fiber $X_F \rightarrow \spec F$.
 Let $\sT_k \subset \sB_k$ be an open subset that satisfies Lemma \ref{lemme:good-model} and such that no vertical component
 of $\sD$ in $\sX_k$ lies above $\sT_k$. We write $V$ for its associated subset in Definition \ref{dfn:def-of-V}.
 Transposing the statement of Proposition \ref{prop:convergence-diviseur-propre-model} in terms of Green functions, we
 have to study the convergence of the sequence of functions
\begin{equation}
  \label{eq:15}
  g_{N} := \frac{1}{\lambda_{1}^{N}} g_{\overline \sD} \circ (s^{\an})^{N}
\end{equation}
over $\sU_\cT^{\an}$ with respect to the boundary topology.
We first prove the convergence away from
$\overline{p_{-}}$ and then, around $\overline{p_{-}}$. Set
\begin{equation}
  \label{eq:8}
  h_{N} = \frac{1}{\lambda_{1}^{N}} g_{\overline \sD} \circ (s^{\an})^{N} - \left(\frac{\mu}{\lambda_{1}}\right)^{N}
  g_{\overline \sD}.
\end{equation}

We can suppose up to blowing up $\sB_k \setminus \cT_k$ (and $\sX_k \setminus \sX_\cT$) that $\cT_k$ admits a boundary
divisor $\overline \sE_\cT$ defined over $\sB_k$. We will still write $\overline \sE_\cT$ for its pullback over
$\sX_k$ and $g_T$ its associated Green function over $\sX_\cT^{\an}$.

\subsection{Convergence away from $\overline{\{p_-\}}$}
\label{subsec:conv-away-from}
Write $B_\K = \sB_k \times_k \spec \K$, $X_\K = \sX_k \times_k \spec K$, $T_\K = \cT_k \times_k \spec K$, $X_T = X_\K
\times_{B_\K} T_\K$ and $U_T = \sU_\cT \times_\cT T_\K$. In particular, $U_T$ is a quasiprojective model of $S_F$ over
$\K$ and $X_T$ is a quasiprojective model of $X_F$ over $\K$.
We will write $g_T$ for the Green function of $\overline \sE_\cT$ over $X_T^{\an}$.
Since $s$ defines an endomorphism $s : X_T \setminus \overline{\{p_{-}\}} \rightarrow X_T \setminus
\overline{\{p_{-}\}}$, we have the induced map on Berkovich spaces
\begin{equation}
  \label{eq:7}
  s^{\an} : \left(X_T \setminus \overline{\{p_{-}\}}\right)^{\an} \rightarrow \left(X_T \setminus
  \overline{\{p_{-}\}}\right)^{\an}.
\end{equation}
We define an $(s^{-1})^{\an}$-invariant open neighbourhood $U^- = \bigsqcup_{v} U^-_{v}$ of $p_{-}$ in $X_T^{\an}$ as
follows.
For every finite place $v$ of $V$, define
\begin{equation}
  \label{eq:9}
  U^-_{v} := \left\{x \in X_{T}^{\an,v} : r_{\sX_{v}}(x) \in \overline{p_-} \right\}.
\end{equation}
By Lemma \ref{lemme:generic-places}, $U^-_v$ is indeed $(f^{-1})^{\an,v}$-invariant.

For every other place $v$ (there is a finite number remaining), we do the following. Recall the definitions of the open
neighbourhood $O^-$ of $\overline p_-$ in $\sX_\cT$ and the functions $z,w, \phi, \psi$ regular over $O^-$ appearing in the
local normal form of $s^{-1}$. We still write $O^-$ for its intersection with $X_T$. 
Let $\sD_\phi, \sD_\psi$ be the Weil divisor of $\phi, \psi$ over $\sX_k$. By the condition
\ref{item:divisor-phi-psi-no-intersection-with-p} of Lemma \ref{lemme:good-model}, the horizontal part of $\sD_\phi,
\sD_\psi$ not intersecting $O^{-}$ do not intersect $\overline{p_-}$. Let $A > 0$ be such that $-A \sE_T \leq
\sD_{\phi, vert}, \sD_{\psi,vert} \leq A \sE_T$, we have that $\sD_\phi + A \sE_T$ is an effective Weil divisor at any point of
$\overline {p_-}$. Let $K^-_v = r_{\sX_v}^{-1} (\overline{p_-} \cap \sX_T) \subset X_T^{\an,v}$, by Lemma
\ref{lemme:effectiveness-on-Green-function} we have that over $K^-_v, \left| \phi \right|_v , \left| \psi \right|_v \leq
e^{A g_T}$.

Now, let $U^-_v \subset (O^-)^{\an,v}$ be the open subset defined by
\begin{equation}
  U^-_v = \left\{ \left| z \right|_v, \left| w \right|_v < \epsilon_v e^{-A g_T} \right\}.
  \label{eq:def-ouvert-U-moins}
\end{equation}
For some $1> \epsilon_v > 0$ small enough such that $U^-_v \subset K^-_v$. This is a neighbourhood of $\overline{ \left\{
p_- \right\}}^{\an,v}$ in $X_T^{\an,v}$. 
If the local normal form of $s^{-1}$ is monomial, i.e of the form \eqref{eq:forme-normale-monomiale}, then
\begin{equation}
  (s^{-1})^* z = \phi z^a w^b
  \label{eq:<+label+>}
\end{equation}
with $a+b \geq 2$ (since $s^{-1}$ must contract $z = 0$ and $w = 0$), therefore over $U^-_v$ we get
\begin{equation}
  \left| (s^{-1})^* z \right|_v = \left| \phi_v \right| \left| z \right|_v^a \left| w \right|_v^b \leq \epsilon_v^2 e^{A
  g_T (1 - (a+b))} \leq \epsilon_v e^{-A g_T}  \label{eq:<+label+>}.
\end{equation}
The same computation works for $(s^{-1})^* w$. If the local normal form of $s^{-1}$ is of the form
\eqref{eq:forme-normale-henon}, then since $a \geq 2$ and $c+d \geq 2$, the same computation works. 
Therefore, $U^-_v$ is $(s^{-1})^{\an, v}$-invariant.

We define $W^- = X_T^{\an} \setminus U^-$.

\begin{lemme}\label{lemme:estimation-away-from-p-minus}
  Write $h = h_1$, then over $W^-$, there exists a constant $C > 0$ such that
  \begin{equation}
    -C g_T \leq h \leq C g_T.
    \label{eq:inequality-g-T}
  \end{equation}
\end{lemme}
\begin{proof}
  First note that by Lemma \ref{lemme:generic-places}, for any place $v \in V[\fin], h \equiv 0$ over $W^-_v$, so
  \eqref{eq:inequality-g-T} holds.
  Let $\pi : Y_F \rightarrow X_F$ be a minimal sequence of point blow-ups such that the lift of $s$ is a regular map $S
  : Y_F \rightarrow X_F$. In particular, $\pi$ is a sequence of point blow ups above $p_-$. We can find a projective
  model $Y_\K$ of $Y_F$ over $\K$ such that $\pi, S$ extend to a birational morphisms $\pi, S : Y_\K \rightarrow
  X_\K$. We also find a projective model $\sY_k$ of $Y_\K$ over $k$ where $\pi, S : \sY_k \rightarrow \sX_k$ extend to
  regular morphism. Since $s_{X_F}^{*} D = \mu D$, we have that the divisor $\frac{1}{\lambda} S^* D -
  \frac{\mu}{\lambda} D$ is a $\pi$-exceptional divisor in $Y_F$. Fix a place $v \not \in V[\fin]$. We first show
  \eqref{eq:inequality-g-T} over $(O^-)^{\an,v} \setminus U^-_v$. Let $z,w$ be the regular functions over $O^-$
  appearing in Lemma \ref{lemme:good-model} and let $E,F$ be the prime divisors in $O^-$ with equation $z = 0$ and $w=0$
  respectively. The functions $- \log \left| z \right|_v, - \log \left| w \right|_v$ are Green functions of $E$ and $F$ over
  $(O^-)^{\an,v}$ respectively. Since $\pi$ is a sequence of blow-ups above $p_- = E \cap F$, there exists a constant $A
  > 0$ such that, in $Y_F$,
  \begin{align}
    -A \pi^* E &\leq \frac{1}{\lambda} S^* D - \frac{\mu}{\lambda} D \leq A \pi^* E\\
    -A \pi^* F &\leq \frac{1}{\lambda} S^* D - \frac{\mu}{\lambda} D \leq A \pi^* F.
    \label{eq:<+label+>}
  \end{align}
  Thus, we can find a constant $B > 0$ such that
  \begin{align}
    -A \pi^* E - B \sE_T &\leq \frac{1}{\lambda} S^* \sD - \frac{\mu}{\lambda} \sD \leq A \pi^* E + B \sE_T\\
    -A \pi^* F - B \sE_T &\leq \frac{1}{\lambda} S^* \sD - \frac{\mu}{\lambda} \sD \leq A \pi^* F + B \sE_T.
    \label{<+label+>}
  \end{align}
  in $\sY_k$ since $\overline \sE_\cT$ is a boundary divisor of $Y_T = Y_\K \times_{B_\K} T_\K$.
  Therefore, looking at Green functions, we have up to replacing $B$ by a higher constant if $v$ is archimedean,
  \begin{align}
    \label{eq:estimate-z}
    A \log \left| z \right|_v - B g_T &\leq h \leq -A \log \left| z \right|_v + B g_T \\
    A \log \left| w \right|_v - B g_T &\leq h \leq -A \log \left| w \right|_v + B g_T
    \label{eq:estimate-w}
  \end{align}
   Now, over $(O^-)^{\an,v} \setminus U^-_v$ we have either
  \begin{equation}
    \log \left| z \right|_v \geq \log \epsilon_v - g_T \text{ or } \log \left| w \right|_v \geq \log
    \epsilon_v - g_T.
    \label{eq:outside-O-minus}
  \end{equation}
  By putting \eqref{eq:estimate-z}, \eqref{eq:estimate-w} and \eqref{eq:outside-O-minus} together, we get that there
  exists $C_1 >0$ such that
  \begin{equation}
    -C_1 g_T \leq h \leq C_1 g_T.
    \label{eq:<+label+>}
  \end{equation}
  Now, by Lemma \ref{lemme:good-model} \ref{item:boundary-no-intersection-with-p} the closure $Z$
  of $U_T \setminus O^-$ in $X_\K$ is a horizontal subvariety that does not intersect $\overline{ \left\{ p_-
  \right\}}$. Therefore, there is a compact neighbourhood of $Z^{\an,v}$ over which we have
\begin{equation}
  - C_2 g_T \leq h \leq C_2 g_T
\label{eq:<+label+>}
\end{equation}
for some constant $C_2 > 0$. We set $C = \max(C_1, C_2)$. Doing this procedure for the finite number of places outside
$V[\fin]$, we get that \eqref{eq:inequality-g-T} holds over $W^-$.
\end{proof}

\begin{prop}\label{prop:convergence-away-from-p-minus}
  The sequence $h_{N}$ converges over $W^{-}$ with respect to the boundary topology to
  \begin{enumerate}
    \item zero if $|\mu| < \lambda_{1}$.
    \item to a continuous function $h^{+}$ if $\mu = \lambda_{1}$ and
      $D = \theta_{X_F}^{+}$ such that $h^+_{|W^-_{V[\fin]}} \equiv 0$ and
      \begin{equation}
        -A g_T \leq h^+ \leq A g_T
        \label{eq:<+label+>}
      \end{equation}
      for some constant $A>0$. Furthermore, if $G^{+} = h^{+} + g_{\overline \sD}$, then
      $G^{+}$ defines a continuous function over $W^{-} \cap U_T^{\an}$ and $G^{+} \circ s^{\an} = \lambda_{1} G^{+}$.
  \end{enumerate}
\end{prop}
\begin{proof}By Lemma \ref{lemme:estimation-away-from-p-minus}, we have that there exists $C > 0$ such that
  over $W^-$, $\left| h \right| \leq C g_T$. We compute
  \begin{equation}
    \label{eq:6}
    h_{N} = \frac{1}{\lambda_{1}^{N-1}} h \circ (s^{\an})^{N-1} + \frac{\mu}{\lambda_{1}} h_{N-1}.
  \end{equation}
  Thus,
  \begin{equation}
    \label{eq:10}
    h_{N} = \sum_{\ell = 0}^{N-1} \frac{\mu^{\ell}}{\lambda_{1}^{N-1}} h \circ (s^{\an})^{N-1-\ell}.
  \end{equation}
  Now the proposition follows since $\left| h \right| \leq C g_{T}$ over $W^{-}$, $s^* \overline \sE_\cT = \overline
  \sE_\cT$ and $W^-$ is $s$-invariant, so the sum in \eqref{eq:10} is absolutely convergent with respect to the boundary
  topology.
\end{proof}
The same proof with $s^{-1}$ yields an open subset $U^{+}$ of $p_{+}$ in
$X_T^{\an}$ such that $h^{N}$ converges with respect to the boundary topology towards a continuous function
$h^{-}$ over $W^+ = X_T^{\an} \setminus U^+$ that satisfies analogous properties as $h^{+}$. In particular, we can
suppose that for every place $v$ outside $V[\fin]$ that
$U_{v}^{+} \cap U_{v}^{-} = \emptyset$. In particular, $U^- \cap U^+ = \emptyset$ and $W^+ \cup W^- = X_T^{\an}$. We can
shrink $U^+_v, U^-_v$ for $v \not \in V[\fin]$ even more such that $G^{\pm}_{|U_{v}^{\pm}} \geq 1$ because $G^{\pm} -
g_{\overline \sD^{\pm}}$ extends to a continuous function over $U^\pm_v$ and $\theta_{X_F}^{\pm}$ is effective.

\subsection{Convergence everywhere}
\label{subsec:conv-everywh}

Since $\supp \theta_{X_F}^- = X_F \setminus S_F$. We can find a boundary divisor $\overline \sD_0$ of $\sU_\cT$ in
$\sX_k$ such that $\sD_0$ is a model of $\theta_{X_F}^+$. Let $g_0$ be the Green function of $\overline \sD_0$. We
have in particular that for every place $v \not \in V[\fin], g_{0,v} > 0$ and that there exists a constant $A >0$ such
that $-A \overline \sD_0 \leq  \overline \sE_\cT \leq A \overline \sD_0$. Define the following constants
    \begin{equation}
    M_0 = \sup_{U_T^{\an}} \left|\frac{h}{g_0}\right|, \quad M_1 = \sup_{U_T^{\an, V[\fin]^c}}
    \left|\frac{g_T}{g_0}\right|,  \quad M_2 = \sup_{W^-} \left|\frac{ h }{g_T}\right|
     \end{equation}
     \begin{equation}
     M_3 = \sup_{U^-_{V[\fin]^c} \bigcap U_T^{\an}} \left|\frac{g_0}{G^-}\right|, \quad M_4 =
      \sup_{U^-_{V[\fin]^c} \bigcap U_T^{\an}} \left|\frac{G^-}{g_0}\right|
  \end{equation}
    where $V[\fin]^c$ is the set of places of $\K$ not in $V[\fin]$.

    \begin{claim}\label{claim:estimate}
      Set $M := \max (M_2 M_1, M_0 M_3 M_4, M_0)$, then for every $k \geq 0$
      \begin{equation}
      -M g_0 \leq h \circ s^k \leq M g_0
        \label{EqInequality}
      \end{equation}
      over $U_T^{\an}$.
    \end{claim}
    \begin{proof}
      We will write $s$ instead of $s^{\an}$ as to avoid heavy notations.
      Let $k \geq 0$ and $x \in U_T^{\an}$ and let $v$ be the place over which $x$ lies. Suppose first that $s^k (x)
      \in W^-_v$.
      If $v \in V[\fin]$, then $h (s^k (x)) = 0$ by Proposition \ref{prop:convergence-away-from-p-minus}
      and \eqref{EqInequality} is obvious. Otherwise we have
      \begin{equation}
        \left|\frac{h(s^k(x))}{g_0 (x)}\right| = \left| \frac{h(s^k(x))}{g_T(s^k(x))} \right| \cdot
        \left| \frac{g_T(x)}{g_0(x)} \right| \leq M_2 M_1
        \label{<+label+>}
      \end{equation}
      and \eqref{EqInequality} is satisfied.

      If $s^k(x) \not \in W^-_v$ then $x, s^k (x) \in U^-_v \subset W^+_v$. If $v \in V[\fin]$, then by
      Proposition \ref{prop:convergence-away-from-p-minus}
      \begin{equation}
        G^-_{|W^+_v} = g_{(\sX_v, \sD^-_v)} = g_{0,v},
      \end{equation}
      thus
      \begin{equation}
        \left|h(s^k (x))\right| \leq M_0 g_0 (s^k (x)) = \frac{M_0}{\lambda_1^k} g_0 (x).
        \label{<+label+>}
      \end{equation}

      Finally, if $v \not \in V[\fin]$, let $y = s^k (x)$, then
      \begin{equation}
        \left|\frac{h(s^k(x))}{g_0 (x)}\right|= \left|\frac{h(y)}{g_0 (s^{-k} (y))}\right| \leq M_4 \left|\frac{h(y)}{ G^-
        (s^{-k}(y))}\right| = M_4 \left|\frac{h(y)}{\lambda_1^k G^- (y)}\right|.        \label{<+label+>}
      \end{equation}
      Thus,
      \begin{equation}
        \left|\frac{h(s^k(x))}{g_0}\right| \leq \frac{M_4}{\lambda_1^k} \left|\frac{g_0(y)}{G^-(y)}\right|
        \left|\frac{h(y)}{g_0(y)}\right| \leq \frac{M_0 M_3 M_4}{\lambda_1^k}
        \label{<+label+>}
      \end{equation}
    \end{proof}
    With this estimate, we have that
    \begin{equation}
      h_N = \sum_{k=0}^{N-1} \frac{\mu^k}{\lambda_1^{N-1}} h \circ s^{N-1-k}
      \label{eq:series}
    \end{equation}
    converges over $U_T^{\an}$ with respect to the boundary topology because the sum in \eqref{eq:series} is absolutely
    convergent with respect to the boundaru topology. If $\left| \mu \right| < \left| \lambda \right|$,
    then $h_N$ converges to zero because in that case.
    \begin{align}
      (\mu \neq 1) \quad &\left| h_N \right| \leq \frac{M}{\lambda_1^{N-1}}\frac{\mu^N - 1}{\mu -1} \left| g_0 \right| \\
      (\mu = 1) \quad &\left| h_N \right| \leq \frac{MN}{\lambda_1^{N-1}} \left| g_0 \right|.
      \label{eq:<+label+>}
    \end{align}

    If $\mu = \lambda_1$ and $\overline \sD = \overline \sD^+$, then we call $\overline{\theta^+}(X_F)$ the limit. It
    satisfies $s^* \overline{\theta^+}(X_F) = \lambda_1 \overline{\theta^+}(X_F)$ and it depends only on $X_F$ by Lemma
    \ref{lemme:convergence-diviseur-verticaux}. Its Green function coincides with the function $G^+$ from Proposition
    \ref{prop:convergence-away-from-p-minus} and we have from Claim \ref{claim:estimate}
    \begin{equation}
      - M \overline \sD_0 \leq \overline{\theta^+}(X_F) - \overline \sD^+ \leq M \overline D_0.
      \label{eq:<+label+>}
    \end{equation}

\subsection{End of proof of Theorem \ref{thm:invariant-adelic-divisor}}
\label{subsec:end-proof-theorem}
We denote now $\lambda_1$ for $\lambda_1(f)$.

\begin{prop}\label{prop:convergence-pour-cartier-pour-s}
  Let $X_F$ be a completion of $S_F$ that satisfies Proposition \ref{prop:dynamics-loxodromic-automorphism} and let $s =
  f^{N_0}$ be an iterate of $f$ that contracts $X_F \setminus S_F$.
  If $D \in \DivInf(X_F)_{\R}$ and $\overline \sD$ is a model adelic extension of $D$ then
  \begin{equation}
    \frac{1}{\lambda_{1}(s)^{N}} (s^{N})^{*} \overline \sD \rightarrow (D \cdot \theta^{-}) \overline \theta^{+}(X_F).
  \end{equation}
\end{prop}
\begin{proof}
If $\lambda_{1} \in \Z$,
then set $D^{-} = 0$. By Proposition \ref{prop:operateur-pull-back}, we can write
\begin{equation}
  \label{eq:13}
  D = a \theta_{X_F}^{+} + b D^{-} + R
\end{equation}
where $a,b \in \R$ and $p_+ \not \in \Supp R$. Intersecting \eqref{eq:13} with $\theta_{X_F}^{-}$ we get $a = D \cdot \theta_{X_F}^{-} = D \cdot
\theta^{-}$. Let $\sX_{k}$ be a projective model of $X_F$ over $k$ where $\sD$ is defined. We assume that there exists a
projective model $B_k$ of $F$ over $k$ with a morphism $\sX_k \rightarrow \sB_k$ with generic fiber $X_F \rightarrow
\spec F$. Write $\sD^+, \sD^-,
\sR$ for the horizontal divisor in $\sX_{k}$ defined by $\theta_{X_F}^+, D^-,
R$ respectively. Then, $\overline \sD$ is of the form
\begin{equation}
  \label{eq:14}
  \overline \sD = (D \cdot \theta^{-}) \overline{ \sD^{+}} + b \overline{ \sD^{-}} + \overline \sR + \overline M
\end{equation}
where $\overline M$ is a model vertical adelic divisor. By
Lemma \ref{lemme:convergence-diviseur-verticaux} and Proposition \ref{prop:convergence-diviseur-propre-model}, we get that
\begin{equation}
  \frac{1}{\lambda_1(s)^N} (s^N)^* \overline \sD \rightarrow (D \cdot \theta^-) \overline \theta^+ (X_F).
  \label{eq:convergence-for-iterate}
\end{equation}
\end{proof}

\begin{prop}\label{prop:convergence-pour-cartier-trivial-intersection}
  Let $X_F$ be any completion of $S_F$ and $D \in \DivInf(X_F)_\R$ such that $D \cdot \theta^- = 0$, then for any model
  adelic extension $\overline \sD$ of $D$ we have
  \begin{equation}
    \frac{1}{\lambda_1^N} (f^N)^* \overline \sD \rightarrow 0.
    \label{eq:<+label+>}
  \end{equation}
\end{prop}
\begin{proof}
  We can suppose that $X_F$ satisfies Proposition \ref{prop:dynamics-loxodromic-automorphism}. Let $s = f^{N_0}$ be an
  iterate of $f$ that contracts the whole boundary $X_F \setminus S_F$, then we have by Proposition
  \ref{prop:convergence-pour-cartier-pour-s}
  \begin{equation}
    \frac{1}{\lambda_1^{N_0 k}} (s^k)^* \overline \sD \rightarrow 0.
    \label{eq:<+label+>}
  \end{equation}
  Therefore, there exists a sequence of positive numbers $\epsilon_k \rightarrow 0$ such that
  \begin{equation}
    - \epsilon_k \overline \sD_0 \leq \frac{1}{\lambda_1^{N_0 k}} (s^k)^* \overline \sD \leq \epsilon_k \overline \sD_0
    \label{eq:epsilon-k}
  \end{equation}
  where $\overline \sD_0$ is a boundary divisor. We can also assume without loss of generality that $- \overline \sD_0
  \leq \overline \sD \leq \overline \sD_0$. Let also $A > 0$ be a constant such that for every $\ell = 0, \dots,
  N_0 - 1$,
  \begin{equation}
    0 \leq \frac{1}{\lambda_1^\ell} (f^\ell)^* \overline \sD_0 \leq A \overline \sD_0.
    \label{eq:constante-A}
  \end{equation}
  For every $k \geq 1$, write $k = n_k N_0 + r_k$ the Euclidian division of $k$ by $N_0$. We have
  \begin{equation}
    \frac{1}{\lambda_1^k} (f^k)^* \overline \sD = \frac{1}{\lambda_1 (s)^{n_k}} (s^{n_k})^*
    \left(\frac{1}{\lambda_1^{r_k}} (f^{r_k})^* \overline \sD\right)
    \label{eq:<+label+>}
  \end{equation}
  and therefore by \eqref{eq:epsilon-k} and \eqref{eq:constante-A} we have
  \begin{equation}
    - A \epsilon_{n_k} \overline \sD_0 \leq \frac{1}{\lambda_1^k}(f^k)^* \overline \sD \leq A \epsilon_{n_k} \overline \sD_0.
    \label{eq:<+label+>}
  \end{equation}
  which shows the result.
\end{proof}

\begin{prop}\label{prop:convergence-pour-cartier}
  Let $X_F$ be any completion of $S_F$. If $D \in \DivInf(X_F)_{\R}$ and $\overline \sD$ is a model adelic extension of $D$, then
  \begin{equation}
    \frac{1}{\lambda_{1}^{N}} (f^{N})^{*} \overline \sD \rightarrow (D \cdot \theta^{-}) \overline \theta^{+}(X_F).
  \end{equation}
\end{prop}
\begin{proof}
  We can suppose that $X_F$ satisfies Proposition \ref{prop:dynamics-loxodromic-automorphism}. Let $s = f^{N_0}$ be an
  iterate of $f$ that contracts the boundary $X_F \setminus S_F$. Then, we have
  \begin{equation}
    \frac{1}{\lambda_1^{N_0 k}} (s^k)^* \overline \sD \rightarrow (D \cdot \theta^-) \overline \theta^+ (X_F).
    \label{eq:convergence-for-iterate}
  \end{equation}
  Now, for every $k \geq 1$, write $k = n_k N_0 + r_k $ the Euclidian division of $k$ by $N_0$, we have
  \begin{equation}
    \frac{1}{\lambda_1^k} (f^k)^* \overline \sD = \frac{1}{\lambda_1^{N_0 n_k}} (s^{n_k})^* \left(
    \frac{1}{\lambda_1^{r_k}} (f^{r_k})^* \overline \sD \right).
    \label{eq:<+label+>}
  \end{equation}
  So to show the proposition, we only need to prove that \eqref{eq:convergence-for-iterate} holds for
  $\frac{1}{\lambda_1^\ell} (f^\ell)^* \overline \sD$ for $\ell = 0, \dots, N_0 -1$. Now, define the model adelic divisor
  \begin{equation}
    \overline \sD_\ell := \overline \sD - \frac{1}{\lambda_1^\ell} (f^\ell)^* \overline \sD.
    \label{eq:<+label+>}
  \end{equation}
  It satisfies $D_\ell \cdot \theta^- = 0$, therefore by Proposition
  \ref{prop:convergence-pour-cartier-trivial-intersection}, we have
  \begin{equation}
    \frac{1}{\lambda_1^{N_0 k}} (s^k)^* \overline \sD_\ell \rightarrow 0
    \label{eq:<+label+>}
  \end{equation}
  and the proposition is shown.
\end{proof}

\begin{prop}
  The adelic divisor $\overline \theta^{+} := \overline \theta^{+}(X_F)$ does not depend on $X_F$. It is strongly nef and
  effective and for every integrable adelic divisor $\overline D$ over $S_F$, one has
  \begin{equation}
    \label{eq:16}
    \frac{1}{\lambda_{1}^{N}} \left(f^{N}\right)^{*} \overline D \rightarrow (D \cdot \theta^{-}) \overline \theta^{+}.
  \end{equation}
\end{prop}
\begin{proof}
  The subset of completions of $S_F$ satisfying Proposition \ref{prop:dynamics-loxodromic-automorphism} is cofinal
in the set of completions of $S_F$. Thus, it suffices to prove that
$\overline \theta^{+}(X_F) = \overline \theta^{+}(Y_F)$ for any completion $Y_F$
above $X_F$ satisfying Proposition \ref{prop:dynamics-loxodromic-automorphism}. Let $\pi : Y_F \rightarrow X_F$ be the
morphism of completions, we have
\begin{equation}
  \pi^* \theta_{X_F}^+ \cdot \theta^- = \pi^* \theta_{X_F}^+ \cdot \theta^-_{Y_F} = \theta_{X_F}^+ \cdot \pi_*
  \theta_{Y_F}^- = \theta_{X_F}^+ \cdot \theta_{X_F}^- = \theta_{X_F}^- \cdot \theta^-.
\end{equation}
Applying Proposition \ref{prop:convergence-pour-cartier} we get that $\overline
\theta^{+}(X_F) = \overline \theta^{+} (Y_F)$.

We show that $\overline \theta^{+}$ is strongly nef and effective. By Goodman's
theorem in \cite{goodmanAffineOpenSubsets1969}, there exists an ample effective
divisor $H$ on $X$ such that $\Supp H = \BD$. Let $\overline H$ be a semipositive and effective extension of $H$, since
$H$ is ample, we have $H \cdot \theta^{-} > 0$ and by Proposition \ref{prop:convergence-pour-cartier} applied with
$\overline H$, we get that $\overline \theta^{+}$ is strongly nef and effective.

Finally, let $\overline D$ be any adelic divisor, let $\overline D_{0}$ be a
boundary divisor. For any $\epsilon >0$, there exists a model adelic divisor $\overline \sD_{\epsilon}$ such that
\begin{equation}
  \label{eq:17}
  \overline \sD_{\epsilon} - \epsilon \overline D_{0} \leq \overline D \leq \overline \sD_{\epsilon} + \epsilon \overline D_{0}.
\end{equation}
Since $f^{*}$ preserves effectiveness, letting $\epsilon \rightarrow 0$ and using the fact that $D_{\epsilon}\cdot \theta^{-} \rightarrow D \cdot \theta^{-}$, we get the result.
\end{proof}

\begin{rmq}\label{rmq:local-setting}
  If $K_v$ is a complete field with respect to an absolute value and $U_{K_v}$ is a normal affine surface over
  $K_v$. We can define the notion of arithmetic divisor over $U_{K_v}$ (see \S 3.6 of \cite{yuanAdelicLineBundles2023}).
  The same proof as in this section shows that if $f$ is a loxodromic automorphism of $U_{K_v}$, then there exists two
  arithmetic divisors $\overline \theta^+, \overline \theta^-$ over $U_{K_v}^{\an}$ unique up to multiplication by a
  positive constant.
\end{rmq}

\begin{prop}\label{prop:properties-green-functions}
  Let $w \in (\spec F)^{\an}$, let $G^+$ be the Green function of $\overline \theta^+$ over $S_F^{\an,w}$, then
  \begin{enumerate}
    \item $G^+ \geq 0$.
    \item $G^+ \circ f^{\an} = \lambda_1 G^+$.
    \item $G^+ (x) = 0$ if and only if the forward $f^{\an}$ orbit of $x$ is bounded (i.e relatively compact in
      $S_F^{\an,w}$).
    \item $G^+$ is plurisubharmonic and pluriharmonic over the set $\left\{ G^+ > 0
      \right\}$.
    \item If $X_F$ is a completion of $S_F$ that satisfies Proposition \ref{prop:dynamics-loxodromic-automorphism}, then
      for any Green function $g$ of $\theta_{X_F}^+, G^+ - g$ extends to a continuous function over $X_F^{\an,w}
      \setminus \left\{ p_- \right\}$.
    \item Similar properties hold for $G^-$ and the set $G^+ + G^- = 0$ is a compact subset of $S_F^{\an, w}$.
  \end{enumerate}
\end{prop}
\begin{proof}
  Items (1) and (2) follow from $\overline{\theta^+}$ being effective and the equality $f^* \overline{\theta^+} =
  \lambda_1 \overline{\theta^+}$. The Green function $G^+$ is plurisubharmonic because $\overline{\theta^+}$ is strongly
  nef. Now, let $v \cM (K)$ be the place over which $w$ lies and let $X_F$ be a completion of $S_F$ that satisfies
  Proposition \ref{prop:dynamics-loxodromic-automorphism}. With the notation of the proof of Proposition
  \ref{prop:convergence-diviseur-model}, we can assume that $v \not \in V[\fin]$. Then, if $U^+_v$ is the open subset
  defined in \eqref{eq:def-ouvert-U-moins}, $U^+_w := U^+_v \cap X_F^{\an,w}$ is a compact $f$-invariant neighbourhood
  of $p_+$ in $X_F^{\an,w}$, indeed ${g_T}_{|X_F^{\an,w}}$ is a constant. Analogously, we have a compact
  $f^{-1}$-invariant neighbourhood $U^-_w$ of $p_-$ in $X_{F}^{\an,w}$ and we can assume that $G^\pm_{U^\pm_w} > 0$ by
  shrinking $U^\pm_w$.

  We show (3), if the forward orbit of $x$ is bounded, then the sequence $(G^+(f^k (x)))_{k \geq 0}$ is bounded by a
  constant $C > 0$ and
  \begin{equation}
    G^+ (x) = \frac{1}{\lambda_1^k} G^+ (f^k (x)) \leq \frac{M}{\lambda_1^k} \xrightarrow[k \rightarrow +
    \infty]{} 0.
  \end{equation}
  Conversely, if the forward orbit of $x$ is not bounded, then it must accumulate to a point $q \in X_F^{\an,w} \setminus
  S_F^{\an,w}$ by compactness. The point $q$ cannot belong to $U^-_w$ because $U^-_w$ is $f^{-1}$-invariant and we would
  get $X \in U^-_w$. Since $U^-_w$ can be taken arbitrary small this is not possible. In particular, $q \neq p_-$ and
  therefore $f(q) = p_+$. Thus, there exists $k_0 \geq 0$ such that $f^{k_0} (x) > 0$ and $G^+(x) =
  \frac{1}{\lambda_1^{k_0}} G^+ (f^{k_0} (x)) > 0$. In fact, this shows that
  \begin{equation}
    \left\{ G^+ > 0 \right\} = \bigcup_{k \geq 0} f^{-k} \left( U^+_v \right).
    \label{eq:carac-G-positif}
  \end{equation}

  To show (5), we can suppose that $g$ is a model Green function of $\theta_{X_F}^+$. We have by Proposition
  \ref{prop:convergence-away-from-p-minus} that $G^+ - g$ extends to a continous function over $X_F^{\an,w} \setminus
  U^-_w$. Since $U^-_w$ can be taken arbitrary small we get the result. For (6), we have that the zero set of $G^+ +
  G^-$ is contained the complement of $f^{-1}(U^+_v) \cup U^-_v$ which is an open neighbourhood of $X_F^{\an,w}
  \setminus S_F^{\an,w}$ so its complement is compact.

  Finally, to show (4), it remains to show that $G^+$ is pluriharmonic over $\left\{ G^+ > 0 \right\}$. Let $H$ be a very
  ample effective divisor in $X_F$ supported outside $S_F$, such a divisor exists by Goodman's result in
\cite{goodmanAffineOpenSubsets1969}. We have by the Hodge index theorem that $\theta^\pm \cdot H > 0$. Let $\overline H$
be the Weil metric of $H$ associated to some set of global generators with Green function $g_{\overline H}$. We can
assume that that $g_{\overline H}$ is pluriharmonic over $U^+_w \cap S_F^{\an,w}$. By Proposition
\ref{prop:convergence-pour-cartier}, we have
\begin{equation}
  \frac{1}{\lambda_1^N} (f^N)^* \overline H \xrightarrow[N \rightarrow +\infty]{} (H \cdot \theta^-) \overline{\theta^+}.
  \label{eq:<+label+>}
\end{equation}
and therefore $\frac{1}{\lambda_1^N}g_{\overline H} \circ f^N$ converges uniformly locally towards $G^+$ over
$S_F^{\an,w}$. Now, if $x \in S_F^{\an,w}$ is such that $G^+ (x) >0$, then for any small enough relatively compact open
neighbourhood $\Omega$ of $x$ in $s_F^{\an,w}$ we have $f^k (\Omega) \subset U^+_w$ for $k$ large enough by
\eqref{eq:carac-G-positif}. Thus, the sequence
\begin{equation}
  \frac{1}{\lambda^n} (g_{\overline H} \circ f^n)_{|\Omega}
  \label{eq:<+label+>}
\end{equation}
is a sequence of pluriharmonic functions converging uniformly to $G^+_{|\Omega}$ which must be pluriharmonic.

Finally, to show (6), we have that $\left\{ G = 0 \right\} = \left\{ G^+ = 0 \right\} \cap \left\{ G^- = 0 \right\}$.
Replace by one of its iterate such that $f^{\pm 1}$ contracts $X_F \setminus S_F$ to $p_\pm$. Then, $W := f^{-1} (U^+_v) \cup
(f^{-1})^{-1} (U^-_v)$ is an open neighbourhood of $(X_F \setminus S_F)^{\an,v}$ and we have that $\left\{ G = 0
\right\} \subset X_F^{\an,v} \subset W$ which is compact as it is a closed subset of a compact subset.
\end{proof}

\begin{dfn}\label{dfn:canonical-adelic-divisor}
  We define a \emph{canonical} adelic divisor of $f$ as any adelic divisor $\overline \theta$ of the form
  \begin{equation}
    \overline \theta := \frac{\overline{\theta^+} + \overline{\theta^-}}{2}.
    \label{eq:<+label+>}
  \end{equation}
  such that $\theta^+ \cdot \theta^- =1$. It is a strongly nef, effective adelic divisor that satisfies $\theta^2 =1$.

  For any other canonical adelic divisor $\overline{\theta '}$ of $f$ we have that there exists $c>0$ such that
  \begin{equation}
    \overline{\theta '} = \frac{c \overline{\theta^+} + \frac{1}{c} \overline{\theta^-}}{2}.
    \label{eq:<+label+>}
  \end{equation}

  In particular for any place $w$, we have $c_1 (\overline \theta)^2_2 = c_1 (\overline{\theta '})^2$ and
  \begin{equation}
    \max(c, 1/c) \overline \theta \leq \overline{\theta '} \leq \max (c,1/c) \overline{\theta}.
    \label{eq:<+label+>}
  \end{equation}
\end{dfn}

\begin{rmq}\label{rmq:normalisation-canonical-divisor}
  One suitable normalisation for $\overline \theta$ is the following. Let $X_F$ be a completion of $S_F$ and $H$ an
  ample divisor supported at infinity. Then, we can assume that $\theta^+ \cdot H = \theta^- \cdot H = c >0$. For such a
  normalisation we have
  \begin{equation}
    \overline \theta = \frac{1}{c} \lim_n \frac{1}{\lambda_1^n} (f^n)^* \overline{\theta^+} + \frac{1}{\lambda_1^n}
    (f^{-n})^* \overline{\theta^-}.
    \label{eq:convergence-canonical-divisor}
  \end{equation}
\end{rmq}

\subsection{Northcott property for $\overline \theta$}\label{subsec:northcott-for-theta}
We conclude this section by showing that the Northcott property for any canonical divisor of $f$ in the arithmetic case
and in the geometric case if $F$ is finitely generated over a finite field $\K$.

\begin{lemme}\label{lemme:model-smaller-than-invariant-divisor}
  Let $X_F$ be a completion of $S_F$, suppose $D \in \Div (X_F)_\R$ and $\overline D$ is a model adelic extension of $D$.
  Suppose $\overline \sD^\pm$ are model adelic extensions of $\theta_{X_F}^\pm$ such that
  \begin{equation}
    \overline D \leq \overline \sD^\pm
    \label{eq:<+label+>}
  \end{equation}
  then, there exists a model vertical divisor $\overline M$ such that
  \begin{equation}
    \overline D \leq \overline \theta^+ + \overline \theta^- + \overline M    \label{eq:<+label+>}
  \end{equation}
\end{lemme}
\begin{proof}
  We can suppose that $X_F$ satisfies Proposition \ref{prop:dynamics-loxodromic-automorphism}.
  Let $U^\pm$ be the open neighbourhoods of $\overline{\left\{ p_\pm \right\}}^{\an}$ constructed in the proof of
  \ref{thm:invariant-adelic-divisor}. Let $\overline \sE_\cT$ be the boundary divisor of $V$ in $\sB_k$. Let $W^{\pm}$ be
  the complement of $U^\pm$ in $U_T^{\an}$. We have that $U_T^{\an} \subset W^+ \cup W^-$ and by Proposition
  \ref{prop:convergence-away-from-p-minus} we have over $W^\mp \setminus \Supp D$
  \begin{equation}
    G^\pm = g_{\overline \sD^\pm} + h^\pm \geq g_{\overline D} + h^\pm
    \label{eq:<+label+>}
  \end{equation}
  with $h^{\pm}$ continuous and $h^{\pm} \equiv 0$ over almost every finite place $v$ and $\frac{\left| h^{\pm}
  \right|}{g_T}$ is bounded. Thus for $A >0$ large enough we have
  \begin{enumerate}
    \item over $W^- \setminus \Supp D, A g_T \geq h^+$,
    \item over $W^+ \setminus \Supp D, A g_T \geq h^-$,
    \item over $W^+ \cap W^- \setminus \Supp D, A g_T \geq h^+ + h^-$.
  \end{enumerate}
  We set $\overline M = A \overline \sE_\cT$. Since $G^\pm \geq 0$ we get
  \begin{equation}
    g_{\overline D} \leq G^+ + G^- + g_{\overline M}
    \label{eq:<+label+>}
  \end{equation}
  over $\sU_\cT^{\an} \setminus \Supp D$ and the lemma is shown.
\end{proof}
We set $\overline \theta = \frac{\overline \theta^+ + \overline \theta^-}{2}$. It is a strongly nef, effective adelic
divisor over $S_F$ and $\theta^2=1$.

\begin{cor}\label{cor:northcott-property}
  Suppose $F$ is finitely generated over a number field or over a finite field and let $\overline{\theta}$ be a
  canonical divisor of $f$.
  Let $\overline H_1, \dots, \overline H_d$ be a big and nef polarisation of $F$, then $h_{\overline
  \theta}^{\overline H}$ satisfies the Northcott property:
  \begin{equation}
    \# \left\{ p \in S_F (\overline F) : \deg F \leq A, h_{\overline \theta}^{\overline H} (p) \leq B \right\} < + \infty.
    \label{eq:<+label+>}
  \end{equation}

  In particular, for every $x \in S_F (\overline F)$, \begin{equation}
    h_{\overline \theta}^{\overline H} (x) = 0 \Leftrightarrow x \in \Per(f).
    \label{eq:<+label+>}
  \end{equation}
\end{cor}
\begin{proof}
  Let $X_F$ be a completion of $S_F$ that satisfies Proposition \ref{prop:dynamics-loxodromic-automorphism}.
  By Goodman's theorem
  \cite{goodmanAffineOpenSubsets1969}, there exists an ample effective divisor $L \in \DivInf(X_F)$ such that $\Supp L =
  X_F \setminus S_F$. Since $\Supp \theta_{X_F}^\pm = X_F \setminus S_F$, we get that there exists $m \geq 1$ such that
  $\frac{1}{m}L \leq \frac{1}{2} \theta_{X_F}^\pm$. Thus,
  there exists a model adelic extension $\overline L$ of $L$ and model adelic extensions $\overline \sD^\pm$ of
  $\theta_{X_F}^\pm$ such that $\frac{1}{2m} \overline L \leq \overline \sD^{\pm}$. By Lemma
  \ref{lemme:model-smaller-than-invariant-divisor}, we have
  \begin{equation}
    \frac{1}{m}h_{\overline L}^{(B; \overline H)} \leq h_{\overline \theta}^{(B; \overline H)} + O(1).
    \label{eq:<+label+>}
  \end{equation}
  over $S_F (\overline F)$ and the Northcott property follows from Theorem \ref{thm:northcott-property} and
  \ref{thm:northcott-property-geometric}.

  If $x \in \Per(f)$, then for every $w \in (\spec F)^{\an}$ we have $G^\pm_w (x) = 0$ by Proposition
  \ref{prop:properties-green-functions} and thus $h_{\overline \theta}^{\overline H} (x) = 0$. Conversely, if
  $h_{\overline \theta}^{\overline H} (x) = 0$, then since $\overline \theta^\pm$ is effective we get $h_{\overline
  \theta^\pm}^{\overline H} (x) = 0$ and therefore for every $k \in \Z$, $h_{\overline
  \theta^\pm}^{\overline H} (f^k(x)) = 0$ by the $f$-invariance of $G^\pm$.
  By the Northcott property the set $\left\{ f^k (x) : k \in \Z \right\}$ must be finite and therefore $x \in
  \Per(f)$.
\end{proof}

\begin{cor}\label{cor:enough-looking-at-h-+}
  Suppose $F$ is a finitely generated field over a number field or over a finite field. Let $\overline H_1, \dots,
\overline H_d$ be a big and nef polarisation of $F$ and let $x \in S_F (\overline F)$. The following are equivalent
\begin{enumerate}
  \item $x \in \Per (f)$.
  \item $h_{\overline \theta}^{\overline H} (x) = 0$.
  \item $h_{\overline \theta^+}^{\overline H} (x) = 0$.
  \item $h_{\overline \theta^-}^{\overline H} (x) = 0$.
\end{enumerate}
\end{cor}
\begin{proof}
  (1) and (2) are equivalent by Corollary \ref{cor:northcott-property}. Write $h^\pm := h_{\overline
  \theta^\pm}^{\overline H}$ such that $h_{\overline \theta}^{\overline H} =: h = h^+ + h^-$. Since $\overline
  \theta^\pm$ are effective, we have that $h^+, h^- \geq 0$. Therefore, $h(x) =0 \Rightarrow h^+ (x) = h^- (x) = 0$. So
  (2) implies (3) and (4). It suffices to show that (3) implies (1) and (4) implies (1). We show it for (3). Suppose
  $h^+ (x) = 0$, then
  \begin{equation}
    h (f^n(x)) = h^+ (f^n(x)) + h^- (f_n(x)) = \lambda_1^n h^+ (x) + \frac{1}{\lambda_1^n} h^- (x) \leq h^- (x).
    \label{eq:<+label+>}
  \end{equation}
  Thus, the sequence $(f^n(x))_{n \geq 0}$ has bounded height $h$ and by Corollary \ref{cor:northcott-property}, this
  sequence is finite, thus $x \in \Per (f)$.
\end{proof}

\section{Periodic points and equilibrium measure}\label{sec:equidistribution-periodic-points}
\subsection{Equidistribution of periodic points}
Let $F$ be a finitely generated field over its prime field. Let $U_F$ be a quasiprojective variety over $F$.
Let $w \in (\spec F)^{\an}$, let $(x_n)$ be a sequence of $U_F (\overline F) \subset U_{F_w} (\overline F_w)$ and let $\mu_w$ be
a measure on $U_F^{\an,w}$. We say that the Galois orbit of $(x_n)$ is equidistributed with
respect to $\mu_w$ if the sequence of measures
\begin{equation}
  \delta(x_n) := \frac{1}{\deg(x_n)} \sum_{x \in \Gal(\overline F / F) \cdot x_n} \delta_x
  \label{<+label+>}
\end{equation}
weakly converges towards $\mu_w$, where $\delta_x$ is the Dirac measure at $x$.

We say that a sequence of points $(x_n)$ of $U_F (\overline F)$ is \emph{generic} if no subsequence
of $(x_n)$ is contained in a strict subvariety of $U_F$. In particular, a generic sequence is
Zariski dense.

\begin{lemme}\label{lemme:generic-sequence}
  Let $U_F$ be a projective variety over a finitely generated field $F$ over its prime field and let $(x_n)$ be a Zariski dense
  sequence of $U_F(\overline F)$, then one can extract a generic subsequence of $(x_n)$.
\end{lemme}
\begin{proof}
  The set of strict irreducible subvarieties of $U_F$ is countable because $F$ is countable.
  Let $(Y_q)_{q \in \N}$ be the set of strict irreducible subvarieties of $U_F$. We construct a
  generic subsequence $(x_q')_{q \in \N}$ as follows. Set $Y'_q = \bigcup_{k \leq q} Y_k$. This is a
  strict subvariety of $U_F$. Let $n(1)$ be such that
  $x_{n(1)} \not \in Y_1 = Y_1'$ and suppose we have constructed $n(1) < \dots < n(q)$ such that $x_{n(i)} \not \in
  Y_i'$. Since $(x_n)$ is Zariski dense, there exists an integer $n(q+1) > n (q)$ such
  that $x_{n(q)} \not \in Y_q'$. This defines an increasing sequence $n(q)$ and we set $x_q ' = x_{n(q)}$, The sequence
  $(x_q ')$ is a subsequence of $(x_n)$ which is clearly generic.
\end{proof}

Let $S_F$ be a normal affine surface over $F$ with $\QAlb (S_F = 0)$. Let $f \in \Aut(S_F)$ be loxodromic. Following
Definition \ref{dfn:canonical-adelic-divisor} we write $\overline \theta_f$ for a canonical divisor of $f$. For every
place $w \in (\spec F)^{\an}$, we write $\mu_{f,w}$ for the equilibrium measure of $\overline \theta_f$ over
$S_F^{\an,w}$. We also write for any polarisation $\overline H$ of $F$, $h_f^{\overline H} := h_{\overline \theta_f}$.

\begin{thm}\label{thm:same-equilibrium-measure-almost-everywhere}
  Let $F$ be a finitely generated field over its prime field. Let $S_F$ be a normal affine surface over $F$ with
  $\QAlb (S_F) = 0$ and let
  $f,g \in \Aut(S_F)$ be loxodromic automorphisms.
  Let $\overline H$ be a polarisation of $F$ defined over $\sB$ satisfying the Moriwaki condition.
  If $\Per(f) \cap \Per(g)$ is Zariski dense, then
  \begin{enumerate}
    \item If $\Gamma \subset \sB$, $\mu_{f,\Gamma} = \mu_{\Gamma,g}$.
    \item If $F$ is finitely generated over a number field $\K$ and $\K \hookrightarrow \C$ is an embedding, then for
      $\mu_{\overline H}$-almost every $b \in \sB(\C)$
      \begin{equation}
        \mu_{f,b} = \mu_{g,b}.
        \label{eq:<+label+>}
      \end{equation}
  \end{enumerate}
\end{thm}
In particular, in the geometric setting we have equality of the equilibrium measures at every place thanks to
Proposition \ref{prop:intersection-not-trivial}.
\begin{proof}
  Let $(x_n)$ be a Zariski dense sequence of $\Per(f) \cap \Per(g)$. By Lemma
  \ref{lemme:generic-sequence}, we can suppose that $(x_n)$ is generic. To apply Theorem \ref{thm:arithm-equid}, we need
  to show that $h_f^{\overline H} (S_F) = h_g^{\overline H}(S_F) = 0$. To
  do that we apply Theorem 5.3.3 of \cite{yuanAdelicLineBundles2023}. Namely, let
  \begin{equation}
    e(S_F, f)^{\overline H} := \sup_{U \subset S_F} \inf_{p \in U} h_{f}^{\overline H} (p)
    \label{<+label+>}
  \end{equation}
  where $U$ runs through open subsets of $S_F$. This quantity is called the \emph{essential minimum} of $\overline
  \theta_f$. Since we have a generic sequence of periodic points, we get $e (S_F, f) = 0$. Theorem 5.3.3 of
  \cite{yuanAdelicLineBundles2023} states that
  \begin{equation}
    e (S_F, f)^{\overline H} \geq h_f^{\overline H} (S_F).
    \label{<+label+>}
  \end{equation}
  Therefore we get $h_g^{\overline H} (S_F) = h_f^{\overline H} (S_F) = 0$. The first assertion follows directly from
  the first assertion of Theorem \ref{thm:arithm-equid}.

  We show the second assertion. Let $\K$ be a number field over which $F$ is finitely generated. Fix a quasiprojective
  model $X_\K$ of $S_F$ with a
  morphism $X_\K \rightarrow B_\K$ where $B_\K$ is a projective model of $F$ over $\K$ over which $\overline H$ is defined.
  Fix an embedding $\K \hookrightarrow \C$ and let $\psi : X_\K (\C) \rightarrow \R$ be a continuous function with
  compact support. Notice that for every $b \in B_\K (\C), \psi_{|X_b}$ also has compact support.
     Let $\epsilon >0$ and define
  \begin{equation}
    U_\epsilon = \left\{ b \in B_\K(\C) : \int_{X_b} \psi_{|X_b} \mu_{f,b} \geq \int_{X_b} \psi_{|X_b} \mu_{g,b} +
    \epsilon\right\}.
    \label{eq:<+label+>}
  \end{equation}
  The set $U_\epsilon$ is measurable. Since $\mu_{\overline H}$ is a finite Radon measure, for every
  $\delta > 0$, there exists a compact subset
  $K_\delta$ and an open subset $T_\delta$ such that $K_\delta \subset U_\epsilon \subset T_\delta \subset
  B_\K(\C)$ and
  \begin{equation}
    \mu_{\overline H}(U_\epsilon) - \delta \leq \mu_{\overline H} (K_\delta) \leq \mu_{\overline H}(T_\delta) \leq
    \mu_{\overline H} (U_\epsilon) + \delta.
    \label{eq:<+label+>}
  \end{equation}
Let ${T'}_\delta \Subset T_\delta$ be a relatively compact open neighbourhood of $K_\delta$ in
$T_\delta$. There exists a continuous function $\phi : B_\K (\C) \rightarrow \R$ such that $\Supp \phi \subset
T_\delta$, $\phi_{|{T'}_\delta} \equiv 1$ and $0 \leq \phi
\leq 1$. Now, by Theorem \ref{thm:arithm-equid}, we have
\begin{align}
  \lim_m \frac{1}{\deg x_m} \sum_{y \in \Gal(\overline F / F) \cdot x_m} \int_{B_\K(\C)} \phi(b) \psi(y(b)) d
  \mu_{\overline H}(b) &= \int_{B^{\an,v}} \left( \int_{X_b}\phi(b) \psi_{|X_b} \mu_{f,b} \right) d \mu_{\overline H}(b) \\
  &= \int_{B_\K(\C)} \left( \int_{X_b} \phi(b) \psi_{|X_b} \mu_{g,b} \right) d \mu_{\overline H}(b).
  \label{eq:egalite-integrale}
\end{align}
And we have that the difference of the two integrals on the right hand side satisfies
\begin{equation}
  0 \geq \epsilon \mu_{\overline H} (K_\delta) - 2 \delta M \geq \epsilon \left( \mu_{\overline H}(U_\epsilon) - \delta
  \right) - 2 \delta M
  \label{eq:<+label+>}
\end{equation}
where $M = \max_{X_\K(\C)} \psi$. Letting $\delta \rightarrow 0$, we get $\mu_{\overline H} (U_\epsilon) = 0$.

Therefore, by taking only rational $\epsilon > 0$ and reversing the role of $f$ and $g$ we get that for $\mu_{\overline
H}$-almost every $b \in B_\K(\C)$
\begin{equation}
  \int_{X_b} \psi_{|X_b} \mu_{f,b} = \int_{X_b} \psi_{|X_b} \mu_{g,b}.
  \label{eq:egalite-integrale-local}
\end{equation}
Now, the set of continuous functions with compact support over $X_\K(\C)$ is separable. Let
$(g_i)$ be a dense sequence, then \eqref{eq:egalite-integrale-local} holds for every $g_i$ and for $\mu_{\overline H}$-almost every
$b \in B_\K(\C)$ and therefore by density $\eqref{eq:egalite-integrale-local}$ holds for every compactly supported
continous function $\phi : X_\K(\C) \rightarrow \R$. Now it is clear by the Stone-Weierstrass theorem that for every
$b \in B_\K(\C)$ the set
\begin{equation}
  \left\{ \phi_{|X_b} : \phi : X_\K(\C) \rightarrow \R, \text{continous with compact support} \right\}
  \label{eq:<+label+>}
\end{equation}
is dense in the set of continous function with compact support from $X_b$ to $\R$.
\end{proof}

Write $G_f$ for the Green function of $\overline \theta_f$. We call the support of $\mu_{f,w}$ the \emph{Julia} set of $f$
and the set $\{G_{f,w} = 0\}$ the \emph{generalised Julia set} of $f$. It is clear that $\Supp \mu_{f,w} \subset
\partial \left\{ G_{f,w} = 0 \right\}$. We show here that $\left\{ G_{f,w} = 0 \right\}$ is the polynomial convex hull
of $\Supp \mu_{f,w}$ for every $w \in (\spec F)^{\an}$. This generalises Lemma 6.3 of
\cite{dujardinDynamicalManinMumford2017}.
\begin{thm}\label{thm:polynomial-convex-hull}
  For every $w \in (\spec F)^{\an}$, $\left\{ G_{f,w} = 0
  \right\}$ is the largest compact subset $J_w$ of $S_F^{\an,w}$ containing $\Supp \mu_{f,w}$ such that for any
  $P \in \OO(S_F)$
  \begin{equation}
    \label{eq:23}
    \sup_{\Supp \mu_{f,w}} |P|_w = \sup_{J_w} |P|_w
  \end{equation}
\end{thm}
\begin{proof}
  Fix a place $w$. The adelic divisor $\overline \theta \in \hat \DivInf (S_F / k)$ ($k = \OO_\K$ or $\K$) induces an
  arithmetic divisor $\overline \theta_w \in \hat \DivInf (S_{F_w})$ over $S_{F_w}$ when restricting to the place $w$ in
  the sense of \cite[\S3.6]{yuanAdelicLineBundles2023}.

  Now we show that $\left\{ G_{f,w}=0 \right\}$ is the largest compact subset satisfying this property.
  If $G_{f,w}(x) > 0$, then suppose for example that $G_{f,w}^+(x) > 0$, then $f^n (x)$ converges towards $p_+$ with the
  notations of the proof of Theorem \ref{thm:invariant-adelic-divisor}. Let $P \in
  \OO(S_F)$ such that $\left| P \right| > 0$ over a sufficiently small neighbourhood of $p_+$. We have
\begin{equation}
  \left| P (f^n (x)) \right| \xrightarrow[n \rightarrow + \infty]{} + \infty.
  \label{eq:<+label+>}
\end{equation}
Let $C_0 = \max_{\Supp \mu_{f,w}} \left| P \right|$, then there exists $N_0 > 0$ such that if we
set $Q = (f^{N_0})^* P$, we get
\begin{equation}
  \left| Q(x) \right| = \left| P (f^{N_0}(x)) \right| > \max_{\Supp \mu_f} \left| Q \right| = \max_{\Supp \mu_f} \left| P
  \right|.
  \label{eq:<+label+>}
\end{equation}
This shows the caracterisation of $\left\{ G_{f,w} = 0 \right\}$.

We now show \eqref{eq:23}.
 Let $P \in \OO (S_F)$. Let $X_{F}$ be a
completion of $S_{F}$ that satisfies Proposition \ref{prop:dynamics-loxodromic-automorphism}. Since the support of
$\theta^{\pm}_{X_{F}}$ is the whole boundary $X_{F} \setminus S_{F}$, there
exists a number $a > 0$ such that $-\div_{X_{F}}(P) \leq a \theta_{X_{F}}$.
By Lemma \ref{lemme:model-smaller-than-invariant-divisor}, there exists $M, M' > 0$ such that
  \begin{align}
    \label{eq:inequality-1}\log \left| P \right| &\leq a (G^+_f + G^-_f) + M \\
    \forall \epsilon << 1, \quad \epsilon g_0 &\leq \frac{a}{2} (G^+_f + G^-_f) + M'
    \label{eq:inequality-2}
  \end{align}
  over $S_F^{\an}$, so that in particular $\div(P) \leq a \theta$. The theorem will follow from the following lemma. 
 
 \begin{lemme}\label{lemme:max-Green-functions}
   Let $P \in \OO (S_F)$ and let $w \in \spec (F)^{\an}$. Let $G_w$ be the Green function of $\overline \theta$ over
   $S_{F,w}$. If $C_0 = \supp_{\mu_{f,w}} \left| P \right|$ and $\epsilon >
   0$, then there for every integer $T > 0$ large enough the function 
   \begin{equation}
     T G_w ' = \max \left(\log^+ \frac{|P|_w}{C_0 + \epsilon}, T G_w\right)
     \label{eq:G-prime}
   \end{equation}
   is a semipositive Green function of $T \theta$ over $S_{F,w}$.
 \end{lemme}

 Suppose the lemma. Then around $\Supp \mu_{f,w}$ we have that $T G_w' = T G_w$ therefore $dd^c (T G_w')^2 = dd^c (T G_w)^2 =
 \mu_{f,w}$ in an open neighbourhood of $\Supp \mu_{f,w}$. But since these two Green functions are Green functions of
 $\Theta$ we have that the total mass of $dd^c (T G_w')^2$ is equal to $\Theta^2$ which is the total mass of $dd^C
 (T G_w)^2$. Since these are both positive measures we get the equality of the measures over $S_{F,w}^{\an}$. But now by
 \eqref{eq:inequality-1} we have that $T G_w = T G_w'$ outside a compact subset of $S_{F,w}^{\an}$. By Theorem C of
 \cite{abboudLocalVersionArithmetic2025} we have that $T G_w = T G_w'$ over $S_{F,w}^{\an}$. In particular checking this
 equality over the set $ \left\{ G_w = 0 \right\}$ we get that $\max_{ \left\{ G_w = 0 \right\}} \left| P \right| \leq C_0 +
 \epsilon$. We get the result by letting $\epsilon \rightarrow 0$.

 \begin{proof}[Proof of Lemma \ref{lemme:max-Green-functions}]
 We work over the place $w$ and drop the index $w$ in the notations. Let $H$ be a very ample divisor over $X_F$ such that $\Supp H = X_{F} \setminus S_{F}$. Using Remark
  \ref{rmq:normalisation-canonical-divisor}, we can assume that \eqref{eq:convergence-canonical-divisor} holds. We have
  that $\Gamma(X_{F}, H) \hookrightarrow \OO(S_{F})$.
Let $P_1, \dots , P_r$ be a set of generators of
$\Gamma(X_{F},H)$ and let $\overline H$ be the semipositive extension of $H$ equipped with the Weil metric induced by
$P_1, \dots, P_r$.
By Theorem \ref{thm:invariant-adelic-divisor} we have that
\begin{equation}
  \frac{1}{c} \frac{1}{\lambda_1^N}g_{\overline H} \circ f^{\pm N} \rightarrow G^\pm_f
  \label{eq:convergence-green-functions}
\end{equation}
over $S_{F}^{\an}$ with respect to the boundary topology. 
We define the sequence of numbers $\lambda_N := \lfloor \lambda_1^N \rfloor$ where $\lfloor x \rfloor$ is the integral
part of $x$. We also have that
\begin{equation}
  \frac{1}{c} \frac{1}{\lambda_N} g_{\overline H} \circ f^{\pm N} \rightarrow G^\pm_f
  \label{eq:<+label+>}
\end{equation}
for the boundary topology.
Let $T$ be an integer (we will specify the value of $T$ later), define the Green function
\begin{equation}
  g_N = \frac{T}{\lambda_N} \left( g_{\overline H} \circ f^N + g_{\overline H} \circ f^{-N} \right).
  \label{eq:<+label+>}
\end{equation}
There exists a sequence $\epsilon_N \rightarrow 0$ such that
\begin{equation}
  - \epsilon_N D_0 \leq c \theta - \frac{1}{\lambda_N} \left( (f^N)^* H + (f^{-N})^* H \right) \leq \epsilon_N D_0
  \label{eq:<+label+>}
\end{equation}
in $\Winf(S_F)$. Since $- \div_{X_F} (P) \leq a \theta$, we have that for $N$ large enough, $P^{\lambda_N} \in
\Gamma \left( T(f^N)^* H + T(f^{-N})^* H  \right) $ whenever $Tc - \frac{a}{2} \geq 1$. Now let $C_0 = \max_{\Supp
\mu_f} \left| P \right|$ and let $\epsilon >0$. There exists $t_n \in F_w^\times$ and $r_n \in \Q$ such that $\left|
t_n \right|^{r_n} \rightarrow C_0 + \epsilon$. Let $a_n$ be a positive integer such that $a_n r_n \in \Z$. 

Let $X_N$ be a completion of $S_F$ above $X_F$ such that the lifts of
$f^N, f^{-N} : X_F \dashrightarrow X_F$ are regular maps $X_N \rightarrow X_F$, we have that
\begin{equation}
  \left(\frac{P^{a_N}}{t_N^{a_N r_N}}\right)^{ \lambda_N } \in \Gamma (X_N, a_N T \left[(f^N)^* H + T(f^{-N})^* H\right])
  \label{eq:<+label+>}
\end{equation}
and the function
\begin{equation}
  g_N ' = \frac{1}{a_N} \max \left( \log^+ \left( \frac{\left| P^{a_N} \right|}{t_N^{a_N r_N}} \right),
  \frac{a_N T}{\lambda_N} \left(g_{\overline H} \circ f^N + g_{\overline H} \circ f^{-N}\right)\right)
  \label{eq:<+label+>}
\end{equation}
is a semipositive model Green function of $\frac{T}{\lambda_N} \left( (f^N)^* H + (f^{-N})^* H \right)$ as it is the
Green function associated to the Weil metric induced by the set of generators
\begin{equation}
  \left\{ \frac{P^{a_N}}{t_N^{a_N r_N}}, (f^N)^* P_1^{T a_N}, \dots, (f^N)^* P_r^{T a_N}, (f^{-N})^* P_1^{T a_N}, \dots,
  (f^{-N})^* P_r^{T a_N} \right\}
  \label{eq:<+label+>}
\end{equation}
of $\Gamma (X_N, T a_N \left((f^N)^* H + (f^{-N})^* H \right))$. It converges with respect to the boundary topology
towards the function $T G'$ from \eqref{eq:G-prime}.
\end{proof}
\end{proof}

\subsection{Proof of Theorem \ref{bigthm:rigidity-periodic-points}}
If $\QAlb (S_F) \neq 0$, then we have shown that $S_F = \G_m^2$ and the theorem is already proven in that case, so we
suppose $\QAlb (S_F) = 0$.

Suppose $\car F = 0$.
Let $S_F$ be an affine surface over  $F$ with two loxodromic
automorphisms $f,g$ such that $\Per(f) \cap \Per(g)$ is Zariski dense. By Corollary
\ref{cor:periodic-points-are-algebraic}, we can suppose that $F$ is finitely generated over $\Q$, let $\K$ be the
algebraic closure of $\Q$ in $F$ and let $\overline H$ be an arithmetic polarisation of $F$ satisfying the Moriwaki
condition defined over a projective model $\sB$.
By Theorem \ref{thm:same-equilibrium-measure-almost-everywhere}, for every $\Gamma \subset \sB$ we have $\mu_{f,\Gamma}
= \mu_{g,\Gamma}$ and for $\mu_{\overline H}$-almost
every $b \in \sB(\C)$ we have $\mu_{f,b} = \mu_{g,b}$. By Theorem \ref{thm:polynomial-convex-hull}, this implies
$\left\{ G_{f,\Gamma} = 0 \right\} = \left\{ G_{g, \Gamma} = 0 \right\}$ and $\left\{ G_{f,b} = 0 \right\} = \left\{
G_{g,b} = 0 \right\}$ for $\mu_{\overline H}$-almost every $b \in \sB(\C)$.
Therefore, by Proposition \ref{prop:moriwaki-height-as-integral}, we have the equality of sets
\begin{equation}
  \left\{ h_f^{\overline H} = 0 \right\} = \left\{ h_g^{\overline H} = 0 \right\}.
  \label{<+label+>}
\end{equation}
And this holds for every $\overline H \in \hat \Pic(F / \OO_\K)_\mod$ satisfying the Moriwaki condition. We conclude by
\begin{prop}
  If $f$ is a loxodromic automorphism of $S_F$, then
  \begin{equation}
    \Per(f) = \bigcap_{\overline H} \left\{ h_f^{\overline H} = 0 \right\}
    \label{eq:<+label+>}
  \end{equation}
  where $\overline H$ runs through every arithmetic polarisation satisfying the Moriwaki condition.
\end{prop}
\begin{proof}
  The inclusion $\subset$ is clear. We show the other one. Let $\mathfrak h_f := \mathfrak h_{\overline \theta_f}$ be
  the vector valued height of $\overline \theta_f$.
  Let $x \in S_F(\overline F)$ be in the set on the right side. Then, $\mathfrak h_f (x) \in \hat \Pic(F / \OO_\K)$ satisfies
  \begin{equation}
    \mathfrak h_f (x) \cdot \overline H^d = 0
    \label{eq:<+label+>}
  \end{equation}
  for every $\overline H$ satisfying the Moriwaki condition. By Proposition \ref{prop:moriwaki-condition} this implies
  that $\mathfrak h_f (x)$ is numerically trivial. Since $\overline \theta_f = \overline \theta^+_f + \overline
  \theta^-_f$ and both are nef, this implies that $\mathfrak h_{\overline \theta^\pm} (x)$ is numerically trivial and by the
  $f$-invariance this holds for every $f^k (x)$ for $k \in \Z$. Now, if we pick a big and nef arithmetic polarisation of
  $F$, we get by Corollary \ref{cor:northcott-property} that the set $\left\{ f^k (x) : k \in \Z \right\} \subset \left\{
  h_{f}^{\overline H} = 0 \right\}$ is finite and thus $x \in \Per(f)$.
\end{proof}

For the geometric case, we suppose $\car F = p > 0$. By Corollary \ref{cor:periodic-points-are-algebraic}, we can always
assume that $F$ is finitely generated over its prime field which is a finite field. If $F$ is a finite field, then we
have $\Per (f) = S_F (\overline F) = \Per (g)$ and
Theorem \ref{bigthm:rigidity-periodic-points} is obvious. We suppose that $F$ has transcendence degree $\geq 1$ over its
prime field. By Proposition
\ref{prop:intersection-not-trivial} and Theorem \ref{thm:same-equilibrium-measure-almost-everywhere} we
have for every $v \in \cM(F), \mu_{f,v} = \mu_{g,v}$ and therefore the equality of sets
\begin{equation}
  \left\{ G_{f,v} = 0 \right\} = \left\{ G_{g,v} = 0 \right\}.
  \label{eq:<+label+>}
\end{equation}
Thus we get by Proposition \ref{prop:moriwaki-height-as-integral} the equality of sets
\begin{equation}
  \left\{ h_f^{\overline H} = 0 \right\} = \left\{ h_g^{\overline H} = 0 \right\}
  \label{eq:<+label+>}
\end{equation}
for any big and nef geometric polarisation $\overline H$ of $F$.
By Corollary \ref{cor:northcott-property}, this implies $\Per(f) = \Per(g)$.

\subsection{Strong rigidity for Hénon maps} We now prove Theorem \ref{bigthm:strong-rigidity-henon}.
Let $f,g$ be two Henon maps over a finitely generated field $F$ over $\Q$. Let $\K$ be the algebraic closure of $\Q$ in
$F$ and let $\overline \theta_f, \overline \theta_g \in \hat \DivInf(\A^2_F/ \OO_\K)$ be the canonical divisors
of $f$ and $g$ such that $\theta_{\P^2}^\pm = L_\infty$ for $f,g$ (this is possible using Remark
\ref{rmq:normalisation-canonical-divisor}). Let $w$ be an archimedean place of $F$ where we have $\left\{G_{g,w} = 0
\right\} = \left\{ G_{f,w} = 0 \right\}$. Such a place exists thanks to Theorem
\ref{thm:same-equilibrium-measure-almost-everywhere} and Theorem \ref{thm:polynomial-convex-hull}. Following the proof
of \cite{dujardinDynamicalManinMumford2017}, we can show that $G_{f,w} ' := \max(G^+_{f,w}, G^-_{f,w})$ is a nonnegative
psh function over $\C^2$ with logarithmic growth and zero set $\left\{ G_{f,w}' = 0 \right\} = \left\{ G_{f,w} = 0
\right\}$. Thus, $G_{f,w}'$ must be the Green-Siciak function of the compact set $\left\{ G_{f,w} = 0 \right\} = \left\{
G_{g,w} = 0 \right\}$ and since it is unique we must have $\max (G^+_f, G^-_f) = \max\left( G_g^+, G_g^- \right)$. The
paragraph after Lemma 6.5 in \cite{dujardinDynamicalManinMumford2017} shows that $f,g$ must share a common iterate.

\bibliographystyle{alpha}
\bibliography{biblio}

\end{document}